\documentclass[letterpaper,11pt,leqno]{amsart}

\usepackage[portrait,margin=3cm]{geometry} 
\usepackage{mathrsfs,xfrac}
\usepackage[
    colorlinks=true,%
    breaklinks,%
    linkcolor=blue,%
    anchorcolor=blue,%
    citecolor=blue,%
    urlcolor=blue
  ]{hyperref} 
\usepackage{amsmath,amssymb,amsthm,amsfonts,amsbsy,latexsym,dsfont,graphicx,color}
\usepackage[foot]{amsaddr}
\usepackage{graphicx}

\usepackage{enumerate}

\newtheorem{thm}{Theorem}[section]
\newtheorem{lem}[thm]{Lemma}
\newtheorem{cor}[thm]{Corollary}

\newtheorem{definition}{Definition}

\newcommand{\half}{\sfrac{1}{2}}
\renewcommand{\le}{\leqslant} 
\renewcommand{\ge}{\geqslant}

\newcommand{\ra}{\rangle}
\newcommand{\la}{\langle}

\providecommand{\1}{\mathds{1}}
\newcommand{\eps}{\varepsilon}

\newcommand{\norm}[1]{\bigl\Vert#1\bigr\Vert}
\newcommand{\abs}[1]{\bigl\vert#1\bigr\vert}

\newcommand{\Real}{\mathds{R}}

\newcommand{\ie}{\emph{i.e.,}}

\newcommand{\equald}{\stackrel{\mathrm{d}}{=}}
\newcommand{\tod}{\stackrel{\mathrm{(d)}}{\longrightarrow}}
\newcommand{\toP}{\stackrel{\mathrm{P}}{\to}}

\def\qed{ \hfill $\blacksquare$}  
\let\ga=\alpha \let\gb=\beta \let\gc=\gamma \let\gd=\delta 
     \let\gl=\lambda      \let\go=\omega   \let\gs=\sigma

\let\gO=\Omega         \let\gS=\Sigma  
                                
\newcommand{\cE}{\mathcal{E}}\newcommand{\cF}{\mathcal{F}}
\newcommand{\cG}{\mathcal{G}}\newcommand{\cH}{\mathcal{H}}

\newcommand{\cM}{\mathcal{M}}\newcommand{\cN}{\mathcal{N}}
\newcommand{\cR}{\mathcal{R}}

\newcommand{\cV}{\mathcal{V}}\newcommand{\cX}{\mathcal{X}}
  
\newcommand{\mv}[1]{\boldsymbol{#1}}

\newcommand{\mva}{\boldsymbol{a}}

\newcommand{\mvi}{\boldsymbol{i}}

\newcommand{\fm}{\mathfrak{m}}

\newcommand{\dN}{\mathds{N}}

\newcommand{\dR}{\mathds{R}}
\newcommand{\dT}{\mathds{T}}

\newcommand{\dZ}{\mathds{Z}} 
\newcommand{\sR}{\mathscr{R}}

\DeclareMathOperator{\E}{\mathds{E}}
\DeclareMathOperator{\pr}{\mathds{P}}

\renewcommand{\th}{ {\textrm{th}} }
\DeclareMathOperator{\var}{Var}
\DeclareMathOperator{\cov}{Cov}

\DeclareMathOperator{\tr}{Tr} 
\DeclareMathOperator{\sech}{sech}

\DeclareMathOperator{\N}{N}
\DeclareMathOperator{\diag}{Diag}
\newcommand{\ol}[1]{\overline{#1}}

\begin{document}

\title{Disordered Monomer-Dimer model on Cylinder graphs}
\author[Dey]{Partha S.~Dey$^\star$}
\author[Krishnan]{Kesav Krishnan$^\dagger$}
%\date{\today}

\address{Department of Mathematics, University of Illinois at Urbana Champaign, 1409 W Green Street, Urbana, Illinois 61801}
\email{$^\star$psdey@illinois.edu, $^\dagger$kesavsk2@illinois.edu, }


\begin{abstract}
	We consider the disordered monomer-dimer model on cylinder graphs $\mathcal{G}_n$, \emph{i.e.}, graphs given by the Cartesian product of the line graph on $n$ vertices, and a deterministic graph. The edges carry i.i.d.~random weights, and the vertices also carry i.i.d.~random weights, not necessarily from the same distribution. Given the random weights, we define a Gibbs measure on the space of monomer-dimer configurations on $\mathcal{G}_n$. We show that the associated free energy converges to a limit, and with suitable scaling and centering, satisfies a central limit theorem. We also show that the number of monomers in a typical configuration satisfies a law of large numbers and a central limit theorem with appropriate centering and scaling. Finally, for an appropriate height function associated with a matching, we show convergence to a limiting function and prove the Brownian motion limit about the limiting height function in the sense of finite-dimensional distributions.
\end{abstract}

\maketitle
%%%%%%%%%%%%%%%%%%%%%%%%%%%%%%%%%%%%%%%%%%%%%%%%%%%
%\setcounter{tocdepth}{1}\tableofcontents
%%%%%%%%%%%%%%%%%%%%%%%%%%%%%%%%%%%%%%%%%%%%%%%%%%%

%%%%%%%%%%%%%%%%%%%%%%%%%%%%%%%%%%%%%%%%%%%
\section{Introduction}
%%%%%%%%%%%%%%%%%%%%%%%%%%%%%%%%%%%%%%%%%%%

Given a finite graph, a matching or monomer-dimer configuration is a collection of pairwise vertex-disjoint edges. Edges in the configuration are called dimers and uncovered vertices are called monomers. Monomer-dimer models were introduced almost a century ago by Roberts~\cite{ROB} to study the adsorption of hydrogen and diatomic oxygen on a tungsten surface and similar physical phenomena. Formally, a monomer-dimer configuration on a graph $G$ with vertex set $V$ and edge set $E$ can be considered as a matching, \ie\ a collection of non overlapping edges $\fm \subseteq E$. We will denote the collection of all matchings on $G$ by $\cM=\cM(G)$. With the monomer activity given by $\nu_v$ for vertex $v\in V$ and with dimer activity $\go_e$ associated to edge $e\in E$, we define the Hamiltonian on $\cM$ as
\begin{align}\label{def:H}
	\cH(\fm)= \sum_{v\notin \fm}\nu_v + \sum_{e\in \fm} \go_e,\text{ } \fm \in \cM,
\end{align}
where $v\notin \fm$ means the vertex $v$ is not covered by any of the edges in the matching $\fm$. With the Hamiltonian given in~\eqref{def:H}, we can define the Gibbs measure, or monomer-dimer model as
\begin{align}\label{def:Gibbs}
	\mu(\fm)= Z^{-1} \exp(\cH(\fm)),\text{ }\fm \in \cM
\end{align}
where
\begin{align}\label{def:Z}
	Z=Z_{G}:=\sum_{\fm\in \cM} \exp(\cH(\fm))
\end{align}
is the partition function for the Gibbs measure. One can introduce an inverse temperature parameter $\gb>0$ in the model, but this can be absorbed into the weights for simplicity.

The rigorous analysis for the monomer-dimer model was spearheaded by Heilmann and Lieb~\cite{HL}, essentially via analyzing a natural recurrence associated with the partition functions. Let $G$ be the original graph and $u,v \in V_{G}$. Let $G^{u}$ and $G^{u,v}$ denote the principal subgraphs obtained on the sequential removal of $u$ and $v$. For every vertex $u \in V$, the associated partition functions satisfy a recurrence relation
\begin{align} \label{def:recur}
	Z_{G}=\exp\left(\nu_u\right)Z_{G^{u}} +\sum_{v\sim u} \exp\left(\go_{(u,v)}\right)Z_{G^{u,v}}.
\end{align}

With constant monomer weights $x$, the partition function is a polynomial in $e^x$, the monomer fugacity. Lee-Yang zeroes are the roots of this polynomial. Heilmann and Lieb showed that these zeroes are strictly imaginary and exist in conjugate pairs. The recurrence~\eqref{def:recur} implies an interlacing condition for the roots, which will be discussed further in Section~\ref{sec:LeeYang}. They used the interlacing and localization result for the zeroes to prove the absence of phase transition. In particular, it implies that the limiting free energy is an analytic function in $x$, assuring the convergence of all the cumulants. For problems of statistical physics, exact solvability means the explicit computation of the associated partition function. Heilmann and Lieb wrote down exact solutions for the one-dimensional graph with free or periodic boundary conditions, the complete graph, and the Bethe lattice, all with constant weights. One can write the partition function $Z_n$ as a weighted adjacency matrix determinant in the one-dimensional case. We will briefly discuss this in Section~\ref{sec:Jacobi}. Using the analysis from~\cite{HL}, Godsil~\cite{CDG} showed that for certain graphs where the variance of the number of edges goes to infinity faster than a specific rate, the number of edges could be scaled appropriately to show Gaussian fluctuations. Lebowitz et.~al.~in~\cite{LEB} extended this result further for a family of probability measures whose partition functions are graph counting polynomials, who in addition prove a local central limit theorem using localization of Lee-Yang zeroes. The notion of Lee-Yang zeroes was introduced in~\cite{LEE1}. Their use to study phase transition behavior and fluctuations is well known, for instance, see~\cite{BIKS}. In particular, for their use to prove central limit theorems, we refer the reader to~\cite{CLT}. 

For the pure dimer or perfect matching model (with no monomers) on planar graphs, exact solutions were previously found by Kasteleyn~\cite{KAST1} and Fisher and Temperley~\cite{TF}, giving rise to Kasteleyn theory~\cite{Kast2}. Kasteleyn proved that the partition function for the dimer model on a surface graph, \emph{i.e.}, a graph embedded in a surface of genus $g$, could be written as the sum of $2^{2g}$ Pfaffians. Using a determinantal kernel for planar bipartite graphs, one can evaluate the partition function and the correlation function between edges. Moreover, one can define a height function associated with a dimer configuration, uniquely defined up to a constant. In~\cite{Rken1}, Kenyon proved convergence to a limiting height function and Gaussian free field fluctuations with Dirichlet boundary conditions about the limiting surface. The dimer model in 2-dimensional lattices is intimately connected to the study of spin structures and the analysis of free fermions~\cite{SPIN}. This result is related to the fact that there is a one-to-one correspondence between monomer-dimer models at non-zero monomer fugacity and the Ising model at a non-zero external field as explored in~\cite{HL}.

In a series of works~\cite{AC2, AC1, AC3}, the authors analyzed the monomer-dimer model in the mean-field setting, that is, where the vertices are exchangeable. In~\cite{AC1}, the authors obtained an exact solution for the model on locally tree-like graphs, such as Galton-Watson Tree and sparse Erd\"os-Renyi random graph. Their methods employ a rigorous version of a statistical mechanics technique known as the cavity method. When the size of the vertex set is large, the removal of a vertex essentially yields a copy of the system. The authors used this idea to prove the convergence of the mean free energy and the limiting monomer density in terms of a distributional fixed point equation solution. In~\cite{AC3}, a solution for the complete graph with constant edge weights and i.i.d.~random vertex weights was given in terms of a fixed point. The analysis in~\cite{AC3} proceeds via the Gaussian integral representation of the partition function. In the mean-field case where the edge weights are the same constant, the problem reduces to computing a one-dimensional integral involving a polynomial of Gaussian variables. The asymptotics can be evaluated using Laplace's method. Finally, in~\cite{AC2}, the authors introduce an imitative potential, an attractive interaction between pairs of adjacent dimers (or pairs of adjacent monomers), which breaks the symmetry of being able to swap a dimer and adjacent monomer and preserving the Hamiltonian and thus induces a phase transition. They characterize the limiting free energy via a variational principle and evaluate the phase diagram with respect to the monomer fugacity and the potential strength.

The problem of counting monomer-dimer configurations is of interest to the computer science community as well. In the pure dimer model on bipartite surface graphs, the Kasteleyn determinant formula enables efficient computation. In stark contrast, counting configurations with non-zero monomer fugacity is computationally intractable, see~\cite{MKJ}. On bipartite graphs, the problem of counting the number of monomer-dimer configurations is equivalent to the problem of evaluating the permanent of a matrix with $\{0,1\}$ entries which is known to be in the $\# P$ class \cite{PERM}. However, reasonably quick probabilistic algorithms for approximating the number configurations have been described~in~\cite{KRS}. They define Glauber dynamics over the space of matchings and show that the Markov Chain has a sufficiently fast mixing property; to be close to stationarity in polynomial time. One can view the Random Assignment Problem as a version of the monomer-dimer model, and several constraint satisfiability problems can be studied using this framework.

In this work, we address the monomer-dimer problem in the context of disordered weights on cylinder graphs. We will prove the convergence of the mean free energy and monomer density and establish central limit theorems for both. There are contributions to the fluctuations from both the ensemble and the environment for which we establish bounds and explicitly characterize the monomer density. Finally, we show that the spatial monomer density displays white noise fluctuations at the level of finite-dimensional distributions.

%%%%%%%%%%%%%%%%%%%%%%%%%%%%%%%%%%%%%%%%%%%
\subsection{Model}
%%%%%%%%%%%%%%%%%%%%%%%%%%%%%%%%%%%%%%%%%%%%

As stated in the introduction, we will be working on cylinder graphs. We now make this notion precise.

\begin{definition}
	Let $H=(V_{H},E_{H})$ be a fixed graph with $|V_{H}|=h$ and $G_n$ be the line graph on $n$ vertices. A cylinder graph $\cG_n$ is given by the graph Cartesian product
	\begin{align}
		\cG_n:=G_n\times H \label{def:Cylgraph}
	\end{align}
	with vertex set $V_{\cG_n}=V_{G_n}\times V_{H}$ and the adjacency relation given by $u=(i,j)\sim v=(i',j')$ if either $i\sim i'$ in $G_n$ and $j=j'$ or $i=i'$ and $j\sim j'$ in $H$.
\end{definition}

For simplicity, we will use $V$ for $V_{\cG_n}$ and $E=E_{\cG_n}$. Given a vertex $u=(i,j)\in V$, $i$ and $j$ will be referred to as the $G_n$ and $H$ components of $u$, respectively. We will denote the total number of vertices in the cylinder graph by
\begin{align}\label{def:N}
	N:=|V_{\cG_n}|=nh.
\end{align}

We will work with the following scheme of weights. The weight associated to the vertex $v \in V$ is given by $\nu_v$ where $\{\nu_v\}_{v\in V}$ are i.i.d.~real-valued random variables. Similarly, the weight associated to $e\in E$ is given by $\go_e$ where $\{\go_e\}_{e\in E}$ are i.i.d.~and real-valued, and independent of the $\nu_v$'s though not necessarily with the same distribution. Unless explicitly mentioned, we will work with the assumption that $\E(|\nu_v|^{2+\eps} + |\go_e|^{2+\eps})$ is finite for some $\eps>0$. With these choice of weights, we define the Hamiltonian $\cH$ for the model as described in~\eqref{def:H}, the corresponding Gibbs measure as in~\eqref{def:Gibbs} and finally the partition function $Z_n:=Z_{\cG_n}$ as in~\eqref{def:Z}.

We are also interested in the behavior of a typical matching $\fm$ chosen from the Gibbs measure $\mu$, on specific sections of the graph corresponding to given ranges of $G_n$--components. We define the \emph{restricted partition functions} as follows.

\begin{definition}[Restricted Partition Function]\label{def:RestZ}
	Let $\cG_{[k:l]}$ denote the principal subgraph of $\cG_n$ generated by the vertices with $G_n$ components in the interval $[k,l]$. The restricted partition function $Z_{[k:l]}$ is defined as the partition function of the monomer-dimer model on $\cG_{[k:l]}$.
\end{definition}

%%%%%%%%%%%%%%%%%%%%%%%%%%%%%%%%%%%%%%%%%%%
\subsection{Main results}

Here we state and briefly explain the main results about the model described in the previous section. 
%%%%%%%%%%%%%%%%%%%%%%%%%%%%%%%%%%%%%%%%%%%
\subsubsection{Limit theorems for the free energy}
%%%%%%%%%%%%%%%%%%%%%%%%%%%%%%%%%%%%%%%%%%%
As is typical for problems in statistical physics, we begin with the analysis of the mean free energy. Due to the disordered environment, the free energy has random fluctuations, which we aim to characterize.

\begin{thm}[Mean and Variance for the log-partition function] \label{thm:Zmeanvar}
	Let $Z_n$ be the partition function as defined in~\eqref{def:Z}. Assume that $\E (\nu_v^{2}+\go_e^{2})<\infty$. There exist constants $f\in\dR, \gs_F\in (0,\infty)$ depending on the distributions of $\go$ and $\nu$, such that
	\begin{align*}
		n^{-1} \log Z_n \toP f \qquad \text{ and }\qquad n^{-1} \var(\log Z_n) \to \gs_{F}^2 \text{ as } n\to \infty.
	\end{align*}
\end{thm}

We can also prove a Gaussian Central limit theorem for the free energy as given below in Theorem~\ref{thm:Zclt}.

\begin{thm}[Central Limit Theorem for the log-partition function]\label{thm:Zclt}
	Assume that $\E(|\nu_v|^{2+\eps}+ |\go_e|^{2+\eps})$ is finite for some $\eps >0$. We have
	\begin{align*}
		n^{-\half}\cdot (\log Z_n - \E\log Z_n) \tod \N(0,\gs_F^2)
		\text{ as } n\to \infty.
	\end{align*}
\end{thm}

%%%%%%%%%%%%%%%%%%%%%%%%%%%%%%%%%%%%%%%%%%%
\subsubsection{Quenched and Annealed Limit of a typical matching}
%%%%%%%%%%%%%%%%%%%%%%%%%%%%%%%%%%%%%%%%%%%

The random variable of central importance is the number of unpaired vertices corresponding to a typical matching $\fm$, which we will denote by $U=U(\fm)$. The number of unpaired vertices occurring in the section $\cG_{[k:l]}$ will be denoted by $U_{[k:l]}$.

\begin{definition}\label{def:GibbsAverage}
	Let $\cM$ be the collection of all matchings on the graph $\cG_n$. Let $\mu$ be the Gibbs measure on $\cM$ as defined in~\eqref{def:Gibbs}. Consider a function $X:\cM \to \dR$. The Gibbs average of $X$ with respect to $\mu$, denoted by $\la X\ra_n$, is defined as
	\begin{align*}
		\la X \ra_n:=\sum_{\fm \in \cM} X(\fm) \mu(\fm).
	\end{align*}
\end{definition}

First, we look at the mean and fluctuation behavior of $U$ and $\la U\ra$.

\begin{thm}[Law of Large Numbers for Unpaired Vertices]\label{thm:Mvar}

	Let $U(\fm)$ denote the number of unpaired vertices in a matching $\fm$ chosen from $\mu$. There exists a constant $u\in [0,1]$ depending on the distributions of $\go$ and $\nu$ such that
	\begin{align*}
		n^{-1} \la U\ra_n \toP u \text{ as } n \to \infty.
	\end{align*}
	Moreover, there exist constants $\gs_Q\in (0,\infty)$ and $\gs_{A}\in [0,\infty)$ depending on the distributions of $\go$ and $\nu$ such that
	\begin{align*}
		n^{-1}(\la U^{2}\ra_n-\la U\ra^{2}_n) & \toP \gs_Q^2
		\text{ and }
		n^{-1}{\var \la U\ra _n} \to \gs_{A}^2 \text{ as } n \to \infty.
	\end{align*}
	Moreover, let the edge and vertex weights be compactly supported and $d_{max}$ denote the maximal degree of $\cG_n$. If $\go_e-2\nu_v<-\log d_{max}\text{ a.s.}$,~then
	\begin{align*}
		\gs_{A}>0.
	\end{align*}

\end{thm}

The exact form of $u$ can be explicitly written in terms of the limiting empirical distribution of the Lee-Yang Zeroes, as given in Section~\ref{sec:LeeYang}. Convergence of the Gibbs average of $U$ also implies convergence of its higher cumulants, particularly the following quenched Central Limit Theorem for $U$.

\begin{thm}[Quenched CLT for $U$]\label{thm:uclt}
	Let $\fm$ be a typical matching chosen from the Gibbs measure $\mu$ as defined in~\eqref{def:Gibbs}. Let $U=U(\fm)$ be the number of unpaired vertices in the matching $\fm$. We have
	\begin{align*}
		n^{-\half} \cdot (U - \la U\ra_n) \tod \N\left(0,\gs^2_Q\right) \text{ as } n \to \infty \text{ in probability. }
	\end{align*}

\end{thm}

Theorem~\ref{thm:uclt} provides a Gaussian central limit theorem for the number of unpaired vertices in the entire matching. We can also examine the behavior of a typical matching on specific sections of $\cG_n$
\begin{thm}[Quenched Joint CLT] \label{thm:jointclt}
	Let $k$ be an integer such that $k/n \to t\in (0,1)$ as $n \to \infty$. Let $\fm$ be a matching chosen from the Gibbs measure $\mu$. Let $U_{[1:k]}$ and $U_{[k+1:n]}$ denote the number of unpaired vertices whose $G_n$--coordinates are in $[1,k]$ and $[k+1:n]$, respectively. We have
	\begin{align*}
		n^{-\half}\left( {{U}_{[1:k]}}-\la {{U}_{[1:k]}}\ra_n, {{U}_{[k+1:n]}}-\la {{U}_{[k+1:n]}}\ra_n\right)\tod \N^2(0,\gs_Q^2\cdot \diag(t,1-t)) \text{ as } n\to \infty
	\end{align*}
	in probability.
\end{thm}

Theorem~\ref{thm:uclt} has an annealed counterpart, where we examine the fluctuations of the Gibbs averaged number of unpaired vertices arising from the environment.

\begin{thm}[Annealed CLT for $\la U\ra_n$]\label{thm:uhatclt}
	Let $\la U\ra_n$ denote the Gibbs average of the number of unpaired vertices in a typical matching under the Gibbs measure defined in~\eqref{def:Gibbs}. We have
	\begin{align*}
		n^{-\half}\cdot (\la U\ra_n - \E \la U\ra_n ) \tod \N\left(0, \gs_{A}^2\right)
		\text{ as }n\to\infty.
	\end{align*}
\end{thm}

\subsubsection{Height function and Brownian Motion}
The quenched and the annealed Central Limit Theorem together establish the fluctuation behavior of not only the total number of unpaired vertices but also the number of unpaired vertices in sections of $\cG_n$. Motivated by the results in~\cite{Rken1}, we can define a height function associated with a matching in order to analyze the typical behavior.
\begin{definition}\label{def:height}
	Let $U_{[k:l]}$ be defined as in Theorem~\ref{thm:jointclt}. We define the height function $\theta_n:[0,1]\to\dR$ as
	\begin{align*}
		\theta_n(t):=U_{[1:\lfloor nt \rfloor]},
	\end{align*}
	and the scaled centered height function as
	\begin{align*}
		\widehat{\theta}_n(t)=n^{-\half}\cdot (U_{[1:\lfloor nt \rfloor]} - ntu) \text{ for } t\in [0,1].
	\end{align*}
\end{definition}

The law of large numbers suggests that the limiting height function is $tu$ with $u$ as in Theorem~\ref{thm:Mvar}. We will prove this in Section~\ref{sec:BM}. We also characterize the distribution of the scaled height function.

\begin{thm}[Brownian Motion Limit]\label{thm:BM}
	Let $\gs^2:=\gs_Q^2+\gs_{A}^2$ where $\gs_Q,\gs_A$ are as defined in Theorem~\ref{thm:Mvar}. We have
	\begin{align*}
		\bigl( \widehat{\theta}_n(t) \bigr)_{t\in[0,1]}\tod \bigl( \gs B_{t}\bigr)_{t\in [0,1]} \text{ as }n\to \infty
	\end{align*}
	in the sense of finite dimensional distributional convergence, in probability, where
	$\left(B_{t}\right)_{t\in (0,1)}$ is a standard Brownian Motion.
\end{thm}

%%%%%%%%%%%%%%%%%%%%%%%%%%%%%%%%%%%%%%%%%%%
\subsubsection{CLT for the Ground State Energy}
%%%%%%%%%%%%%%%%%%%%%%%%%%%%%%%%%%%%%%%%%%%

Our techniques can be adapted to study the Ground State Energy ,
\[
M_n:= \max_{\fm\in \cM}\cH(\fm),
\]
\ie\ the maximum value of the Hamiltonian over all matchings. This may be regarded as the zero-temperature version of the problem considered in this article. We have the following result.

\begin{thm}\label{thm:gse}
 Let $\E(\go_e^2+\nu_v^{2}) < \infty$. There exist $m\in \dR$ and $\gs_{M}\in [0,\infty)$, such that 
 \[
 n^{-1}M_{n}\toP m \text{ and } n^{-1}\var M_{n} \to \gs_{M}^{2} \text{ as } n\to \infty. 
 \]
 Moreover, if 
 $
 \E( |\go_e|^{2+\epsilon}+|\nu_{v}|^{2+\epsilon})
 <\infty
 $ 
 for some $\eps>0$, we have 
 \[
 n^{-\half}\cdot (M_{n}-\E M_{n}) \tod N(0,\gs_{M}^{2}).
 \]
\end{thm}

Note that our methods cannot be easily adapted to characterize the scaling limit of a ``optimal'' matching, as we cannot use exponential tilting. Moreover, the optimal matching may not be unique unless we assume continuous distributions for $\go$ and $\nu$. On the issue of the limiting variance $\gs_{M}$, it is clear that a necessary condition for the variance to be non-degenerate is that $\pr(\go_{e}-\nu_{v}-\nu_{u}>0)>0$, otherwise the empty matching is optimal. The author in~\cite{CHAT} establishes some general conditions for obtaining fluctuation lower bounds, and one of the techniques can be adapted for the ground state energy under appropriate conditions on the edge and vertex weights. 

%%%%%%%%%%%%%%%%%%%%%%%%%%%%%%%%%%%%%%%%%%%
\subsection{Heuristics }
%%%%%%%%%%%%%%%%%%%%%%%%%%%%%%%%%%%%%%%%%%%
When the monomer fugacity is non-zero, we do not expect long-range correlations between edges or unpaired vertices. This result is easier to see in the one-dimensional case where $|V_{H}|=1$, the presence of an unpaired vertex essentially renders the graph disjoint. The Gibbs measure can be regarded as a product measure of the same defined on the two smaller pieces. The absence of long-range correlations is the mechanism for the central limit theorems and the Brownian motion limit; far apart sections appear as though they are independently sampled. In the cylinder graph case, the presence of a single vertex does not disconnect the graph. However, we may still express the random variables we are trying to prove central limit theorems for as the sums of related i.i.d.~random variables with a perturbation. This decomposition is possible explicitly because of the pseudo-1-dimensional structure. Recall from Definition~\ref{def:RestZ} that $Z_n$ is the partition function associated to $\cG_n$ and $Z_{[1:k]}$ and $Z_{[k+1:n]}$ are the restricted partition functions corresponding to $\cG_{[1:k]}$ and $\cG_{[k+1:n]}$ respectively. We have the trivial bound
\begin{align*}
	Z_{[1:k]}\cdot Z_{[k+1:n]}\le Z_n.
\end{align*}

Let us enumerate the $h$ many edges corresponding to the layer joining $\cG_{[1:k]}$ and $\cG_{[k+1:n]}$ as $e_{k,1},e_{k,2}\ldots e_{k,h}$. We denote the collection of these edges by $\cE_{k}$. Let the weight of edge $e_{k,i}$ be denoted as $\go_{k,i}$ for $i=1,2,\ldots,h$. The vertices adjacent to the edge $e_{k,i}$ are denoted by $(k,i)$ and $(k+1,i)$; and their weights by $\nu_{k,i}$ and $\nu_{k+1,i}$ respectively. Suppose a subset of edges $A \subseteq \cE$ is present in a matching. We denote by $Z_{[1:k]}^{A}$ and $Z_{[k+1:n]}^{A}$ the restricted partition functions on the respective pieces with the vertices incident to $A$ excluded. Applying the recursion defined in~\eqref{def:recur} we clearly have
\begin{align*}
	Z_n=\sum_{A\subseteq \cE}Z_{[1:k]}^{A}Z_{[k+1:n]}^{A}\prod_{e_{k,i} \in A} \exp (\go_{k,i}-\nu_{k,i}-\nu_{k+1,i}).
\end{align*}
In terms of the free energy, we have
\begin{align}\label{def:Decomp}
	\log Z_n=\log Z_{[1:k]} + \log Z_{[k+1:n]} +R_{n,k},
\end{align}
where $R_{n,k}$ is the error arising in the partition function from disconnecting $\cG_n$ into the two components, \ie\
\begin{align}\label{def:error}
	R_{n,k}=\log \left( \sum_{ A \subseteq \cE}\prod_{i : e_{k,i} \in A} e^{\go_{k,i}-\nu_{k,i}-\nu_{k+1,i}}\cdot \frac{Z^{A}_{[1:k]}}{Z_{[1:k]}}\cdot \frac{Z^{A}_{[k+1:n]}}{Z_{[k+1:n]}}\right).
\end{align}
By construction, $R_{n,k}$ is positive. Key to our analysis are moment bounds for $R_{n,k}$. We will show that $R_{n,k}$ and related similar random variables are bounded in norm $p=2+\eps$ for some $\eps>0$.

There is a natural gauge symmetry associated with the monomer-dimer model, which will enable us to transform the weights to make the quenched analysis easier later.

\begin{lem}\label{lem:gauge}
	Consider a vertex $v$ which has weight $\nu_v$ and all adjacent edges $e=(v,w)$ with weights $\go_{(v,w)}$ for all $w\sim v$. Let $y\in \dR$ be a constant. The transformation $\nu_v\to \nu_v+y$ and $\go_{(v,w)}\to \go_{(v,w)}+y$ for all $u \sim v$ leaves the Gibbs measure invariant.
\end{lem}
\begin{proof}
	Let $\cH^y(\fm)$ denote the Hamiltonian with the transformed weights. Note that any matching $\fm$ includes exactly one edge adjacent to $v$, or leaves $v$ unpaired. Thus $\cH^y(\fm)=\cH(\fm)+y$ for all $\fm$. Moreover for the transformed partition function $Z_n^y$, we have
	\begin{align*}Z^y_n=\sum_{\fm \in \cM} \exp \cH^y(\fm)=\exp(y)\sum_{\fm \in \cM}\exp(\cH(\fm))=\exp(y)Z_n.
	\end{align*}
	With $\mu^y$ denoting the transformed Gibbs measure, for all $\fm\in \cM$ we have
	\begin{align*}\mu^y(\fm)=\frac{\exp(\cH^y(\fm))}{Z^y_n}=\frac{\exp(y)\exp(\cH(\fm))}{\exp(y)Z_n}=\mu(\fm).
	\end{align*}
	In particular, the Gibbs measure stays invariant.
\end{proof}
This symmetry will be very useful for us; when carrying out the quenched analysis, we will pass the vertex randomness onto the edges while preserving the Gibbs measure and thus all relevant statistical quantities.

%%%%%%%%%%%%%%%%%%%%%%%%%%%%%%%%%%%%%%%%%%%
\subsection{Roadmap}
%%%%%%%%%%%%%%%%%%%%%%%%%%%%%%%%%%%%%%%%%%%

This article is structured as follows. In Section~\ref{sec:technical} we state relevant technical results, in particular, a subadditive theorem due to Hammersley and an $L^{p}$ bound due to Rosenthal. In Section~\ref{section:MFE}, we prove Theorem~\ref{thm:Zmeanvar}, the convergence of the mean free energy, using a combination of subadditivity and variance control. We then establish convergence of the scaled variance and prove Theorem~\ref{thm:Zclt}, a Gaussian central limit theorem for the free energy. We conclude this section with the proof of Theorem~\ref{thm:gse}, the asymptotic behavior of the ground state energy is characterized. Section~\ref{sec:LeeYang} is dedicated to the analysis of the Lee-Yang zeroes and their interlacing property, which is further used to extract bounds on the cumulants of $U$, as well as joint cumulants of $U_{[i:j]}$'s on disjoint sections. We use these bounds in Section~\ref{sec:Ulims} to prove Theorem~\ref{thm:Mvar} which characterizes the mean and variance behavior of $U$, and Theorem~\ref{thm:uclt}, a Gaussian central limit theorem for $U$. We also provide explicit characterizations for the quenched and annealed contributions to the fluctuations. The culmination of these results is given in Section~\ref{sec:BM}, where we prove Theorem~\ref{thm:BM}, convergence of the scaled number of unpaired vertices within a section to Brownian motion. In Section~\ref{sec:Jacobi}, we discuss the connection of the one-dimensional model to the study of Jacobi matrices. We conclude with Section~\ref{sec:disc}, where we describe exciting problems for future consideration.

%%%%%%%%%%%%%%%%%%%%%%%%%%%%%%%%%%%%%%%%%%%
\section{Technical Results and Notation}\label{sec:technical}
%%%%%%%%%%%%%%%%%%%%%%%%%%%%%%%%%%%%%%%%%%%
Two analytic results are of prime importance to us, a subadditive theorem as proved by Hammersley and an $L^{p}$ bound for sums of i.i.d. random variables as proved by Rosenthal. Both are stated below as used here.

\begin{thm}[Hammersley~\cite{HAM}]\label{thm:subadditive}
	Let $a_n$ and $b_n$ be sequences such that \begin{align*}a_{n+m}\le a_n+a_{m} +b_{n+m}.\end{align*}
	A sufficient condition for $\frac{a_n}{n}$ to converge to limit $\ell < \infty$ is \begin{align*}\sum_{n\ge 1}\frac{|b_n|}{n^2}\le \infty.\end{align*}
\end{thm}
This condition is necessary as well; however, for our purposes, the sufficiency is adequate. Essentially the error terms must have a growth rate strictly lower than linear order in $n$.
\begin{lem}[Rosenthal~\cite{ROS}]\label{thm:rosenthal}
	Let $p>2$ be fixed and $X_1, X_{2},\ldots, X_n$ be i.i.d.~mean zero random variables with $\norm{X_1}_{p}=A$ and $\norm{X_1}_{2}=B$. Then, there exists a finite positive constant $C_p$, depending only on $p$, such that
	\begin{align*}
		\norm{ X_1+X_2+\cdots+X_n }_p \le C_p n^{\half}.
	\end{align*}
\end{lem}
We will repeatedly encounter logarithmic derivatives, and the following formula due to di Bruno is helpful to bound them from above.
\begin{lem}[di Bruno's Formula]\label{lem:MomentCumulant}
	Let $f:\dR \to \dR_{+}$ be $n$--times differentiable, and let $g(t)=\log f(t)$. Let the $n^{th}$ order derivatives of $f$ and $g$ be denoted by $f^{(n)}$ and $g^{(n)}$, respectively. We have,
	\begin{align*}
		 & g^{(n)}(t)=\sum_{\substack{m_1,m_2,\ldots,m_n\ge 0 \\ m_1+2m_2+\cdots+nm_n = n }} (-1)^{\sum_{j=1}^n m_j-1} \cdot \frac{n!\cdot \left(\sum_{j=1}^n m_j\right)!}{\prod_{j=1}^n m_j!\cdot j!^{m_j}}\cdot 
		\prod_{j=1}^n\left(\!\frac{f^{(j)}(t)}{f(t)}\!\right)^{m_j}.
	\end{align*}
\end{lem}
This formula is also referred to as the moment cumulant formula in probability, as it captures how the moments and cumulants of a random variable are related.
%%%%%%%%%%%%%%%%%%%%%%%%%%%%%%%%%%%%%%%%%%%
\subsection{Notation and Assumptions}
%%%%%%%%%%%%%%%%%%%%%%%%%%%%%%%%%%%%%%%%%%%
For the rest of this article, we will assume the following:
\begin{enumerate}

	\item The degree of a vertex $v$ in any graph $\cG$ will be denoted by $\text{deg}(v)$.

	\item For a graph $\cG$, $d_{max}$ will denote the maximal degree over all the vertices.

	\item Let $X=X(\mathfrak{m})$ denote a random variable depending on a monomer-dimer configuration. We will use $\la X\ra _n$ to denote the average with respect to the Gibbs measure $\mu$and $ \E X$ will denote the global average.
	\item When $X$ is real valued, $\|X\|_{p}$ will denote $\left(\E |X|^{p}\right)^{\frac{1}{p}}$.

	\item $X_n\toP X$ will denote random variables $X_n$ converging to $X$ in probability. $X_n\tod X$ will denote random variables converging to $X$ in distribution.
	\item For $x$ and $y\in \dR$, $x\land y$ denotes $\min\{x,y\}$ and $x\lor y$ denotes $\max\{x,y\}$.
	\item For vectors $\mv\xi$ and $\mv\zeta$, $\mv\xi\cdot \mv\zeta$ denotes the dot/inner product.
	\item For a matrix $\mv{A}$, $\mv{A}^{T}$ will denote the transpose.
	\item For matrices $\mv{A}$ and $\mv{B}$, $\mv{A}\mv{B}$ denotes the matrix product.
	\item $\diag(a_1,a_2,\ldots,a_k)$ will denote the $k\times k$ diagonal matrix with $i$-th diagonal entry given by $a_i$ for all $i$.
\end{enumerate}

%%%%%%%%%%%%%%%%%%%%%%%%%%%%%%%%%%%%%%%%%%%
\section{Free Energy: Mean and Fluctuation Analysis}\label{section:MFE}
%%%%%%%%%%%%%%%%%%%%%%%%%%%%%%%%%%%%%%%%%%%
\subsection{Mean and Variance Convergence}

In this section, we will establish bounds on the fluctuation of the mean free energy and its convergence. Before this, we bound the size of the error term $R_{n,k}$ encountered in~\eqref{def:Decomp}.

\begin{lem}\label{lem:errbound}
	Let $p\ge 1$ be fixed. Assume that, $\norm{\go}_p+ \norm{\nu}_p<\infty$. We have
	\begin{align*}
		\norm{R_{n,k}}_p
		\le h(1+\norm{\go}_p+4\norm{\nu}_p).
	\end{align*}
\end{lem}

\begin{proof}
	Key to this lemma is the fact that $R_{n,k}>0$. Note that with $\cE$ and $A$ as introduced in~\eqref{def:error} we have,
	\begin{align*}
		Z_{[1:k]}^{A}\le Z_{[1:k]}\prod_{e_{k,i\in A}}\exp(-\nu_{k,i}-\nu_{k+1,i}).
	\end{align*}
	We combine this with the explicit expression for $R_{n,k}$ provided in~\eqref{def:error}. Thus
	\begin{align*}
		R_{n,k}
		\le \log \bigl( 1+ \sum_{A\neq \phi} \prod_{e_{k,i}\in A}\exp(\go_{k,i}-2\nu_{k,i}-2\nu_{k+1,i})\bigr).
	\end{align*}
	Applying the multinomial theorem, we get
	\begin{align*}
		R_{n,k}
		 & \le \log \prod_{i=1}^{h} \bigl(1+\exp(\go_{k,i}-2\nu_{k,i}-2\nu_{k+1,i})\bigr)
		\le \sum_{i=1}^{h} \bigl(1+ \abs{\go_{k,i}-2\nu_{k,i}-2\nu_{k+1,i}}\bigr).
	\end{align*}
	The last inequality follows from $\log (1+e^{z})\le 1+|z|$ for all $z\in \Real$. The proof follows from	the triangle inequality for norms.
\end{proof}

\begin{lem}\label{lem:freeenergy_mean}
	Let $Z_n$ be the partition function as defined in~\eqref{def:Z}. Assume that $\E(|\go|+|\nu|)<\infty$. There exists a finite constant $f$ such that
	\begin{align*}
		n^{-1} \E \log Z_n \to f \text{ as } n\to\infty.
	\end{align*}
\end{lem}
\begin{proof}
	First, assuming convergence we show that $|f|< \infty$.
	Let $v\in V_{\cG}$. This vertex can either be paired with an adjacent vertex or unpaired. The contribution of $v$ to any configuration can be bounded above by $1+\exp {\nu_v}+\sum_{w\sim v}\exp{\go_{(w,w)}}$. The all monomer configuration has energy given by $\sum_{v\in V}\nu_v$. These yield the upper and lower bounds
	\begin{align*}
		\sum_{v\in V}\nu_v \le \log Z_n \le \sum_{v\in V} \log\bigl(1+\exp (\nu_v)+\sum_{w\sim v}\exp{\go_{(v,w)}}\bigr).
	\end{align*}
	The expectation of the above expression is obviously bounded given the moment condition on $\go_e$ and $\nu_v$. Now, we recall the decomposition introduced in~\eqref{def:Decomp}. We clearly have
	\begin{align*}\E\log Z_n\le \E\log Z_{[1:k]} +\E\log Z_{[k+1:n]} + \norm{R_{n,k}}_1. \end{align*}
	It is sufficient for the first moments of $\nu$ and $\omega$ to exist to guarantee $\norm{R_{n,k}}_1$ being finite and uniformly bounded. Observing that
	\begin{align*} \E \log Z_{[1:k]}=\E \log Z_{k} \text{ and } \E \log Z_{[k+1:n]}=\E \log Z_{n-k} \end{align*}
	immediately yields that $\E\log Z_n$ satisfies the hypothesis of Theorem~\ref{thm:subadditive} with the corresponding $b_n$ being given by $\sup_{1\le k \le n}\norm{R_{n,k}}_1$ which is uniformly bounded above by a constant. This completes the proof of convergence.
\end{proof}

To extend the convergence of $n^{-1}\E\log Z_n$ as $n\to \infty$ to that of the random variable $n^{-1}\log Z_n$ requires an upper bound on the order of the fluctuations. We establish the following bound on the variance.
\begin{lem}\label{lem:efron}
	There exists a finite constant $C$ such that $\var (\log Z_n) \le C n$ for all $n$.
\end{lem}

\begin{proof}
	The proof will proceed via the Efron-Stein inequality. Let $e=(u,v)$ denote an edge, let $\go_e$ the corresponding edge weight, and let $\nu_u$ and $\nu_v$ be the corresponding vertex weights. Let $Z_n^{e}$ denote the partition function obtained when $\go_e$ is replaced by an independent copy $\go_e^{'}$. Analogously, let $Z_n^{v}$ denote the partition function obtained by replacing $\nu_v$ with independent copy $\nu_v^{'}$
	\begin{align*}\var (\log Z_n)\le \sum_{e\in E} \E \left ( \log Z_n^{e}-\log Z_n\right)^{2} +\sum_{v\in V}\left( \log Z_n-\log Z_n^{v}\right)^{2}.\end{align*}
	For an edge $e=(u,v)$, let $\cG_n^{e}$ denote the graph $G_n$ with only the edge $e$ removed. Let $\cG_n^{u,v}$ denote the principal subgraph obtained on removing vertices $u$ and $v$. We now introduce
	\begin{align*}
		\ga_e:=Z_{\cG_n^{e}} \text{ and } \gb_e:=Z_{\cG_n^{u,v}}.
	\end{align*}
	Clearly, we have
	\begin{align*}
		Z_n=\ga_e +\gb_e\exp(\go_e).
	\end{align*}
	On taking the difference of $\log Z_n$ and $\log Z_n^{e}$,
	\begin{align*}
		\log Z_n^{e}-\log Z_n
		&=\log \left(\ga_e + \gb_e \exp({\go_e'}) \right)-\log\left(\ga_e +\gb_e \exp({\go_e})\right)\\
		&=\int_{\go_e'}^{\go_e}\frac{\gb_e e^z}{\ga_e +\gb_e e^z}\, dz.
	\end{align*}
	On applying the triangle inequality, we get
	\begin{align*}
		|\log Z_n^{e}-\log Z_n|\le |\go_e'-\go_e|.
	\end{align*}
	Now for the vertices, given a vertex $v$ let $\cG_n^{v}$ denote the principal subgraph obtained on the removal of $v$ from $\cG_n$. With two vertices $u$ and $v$, we denote $\cG_n^{u,v}$ to be the principal subgraph obtained on the removal of $u$ and $v$. We will introduce
	\begin{align*}
		\widehat{\ga}_v:=\sum_{u\sim v} \exp\left(\go_{(u,v)} \right)Z_{\cG_n^{u,v}} \text{ and }\widehat{\gb}_v:=Z_{\cG_n^{v}}.
	\end{align*}
	Recurrence~\eqref{def:recur} then yields
	$Z_n=\widehat{\ga}_v+\exp(\nu_v)\widehat{\gb}_v.$
	In a process identical to bounding the influence of the edge weights, we find
	$|\log Z_n^{v}-\log Z_n|\le |\nu_v'-\nu_v|.$
	Squaring, and adding over all $e\in E_{\cG_n}$ and $v \in V_{\cG_n}$ completes the proof.
\end{proof}

Using Chebyshev's inequality with Lemma~\ref{lem:efron}, we get the following obvious corollary. 

\begin{cor}\label{cor:Zmean}
	Under the assumption of Theorem~\ref{thm:Zmeanvar}, we have
	\[
		n^{-1}\left( \log Z_n-\E\log Z_n\right)\toP 0 \text{ as } n\to\infty..
	\]
\end{cor}

In order to characterize the limiting fluctuations, we require more steps. We need to show that the limiting variance exists, that it is non-degenerate and then finally prove the central limit theorem. We proceed in this sequence.

\begin{lem}\label{lem:zvarconv}
	Assume that $\E(\go^2+\nu^2)<\infty$. We have
	\begin{align*}
		n^{-1} \var (\log Z_n)\to \gs_{F}^2 \text{ as } n\to \infty.
	\end{align*}
\end{lem}
\begin{proof}
	We begin with
	\begin{align*}
		\log Z_n=\log Z_{[1:k]}+\log Z_{[k+1:n]} + R_{n,k}.
	\end{align*}
	Recall that for a random variable $X$, we denote its centered version as
	\begin{align*}
		\ol{X}:=X-\E X.
	\end{align*}
	The variance of $X$ is then given by $\norm{\ol{X}}_{2}^{2}$. We have
	\begin{align*}
		\ol{\log Z_n}
		=
		\ol{\log Z_{[1:k]}}+\ol{\log Z_{[k+1:n]}} + \ol{R_{n,k}}.
	\end{align*}
	Applying the Cauchy-Schwarz inequality and independence between disjoint blocks, we get
	\begin{align*}
		\norm{\ol{\log Z_n}}_{2}^{2}\le \norm{\ol{\log Z_{k}}}_{2}^{2} + \norm{\ol{\log Z_{n-k}} }_{2}^{2} +\norm{\ol{R_{n,k}}}_{2}(\norm{\ol{\log Z_{k}}}_{2}+\norm{\ol{\log Z_{n-k}}}_{2}+\norm{\ol{R_{n,k}}}_{2}).
	\end{align*}
	Applying Lemmas~\ref{lem:errbound} and~\ref{lem:efron}, we have
	\begin{align*}\norm{\ol{\log Z_n}}_{2}^{2}\le \norm{\ol{\log Z_{k}}}_{2}^{2}+\norm{\ol{\log Z_{n-k}}}_{2}^{2}+C(\sqrt{k}+\sqrt{n-k})\end{align*}
	Applying the concavity of the square root, we then have
	\begin{align*}
		\norm{\log Z_n}_{2}^{2}\le \norm{\log Z_{k}}_{2}^{2}+\norm{\log Z_{n-k}}_{2}^{2}+C'\sqrt{n}.
	\end{align*}
	Thus, the sequence $\sigma_{F,n}^{2}=\norm{\ol{\log Z_n}}_{2}^{2}$ satisfies the hypothesis of the subadditive lemma, application of which completes the proof of convergence.
\end{proof}

\begin{lem}\label{lem:zvarlowerbound}
	There exists a constant $c>0$ such that $ \var (\log Z_n) \ge cn$ for all $n$.
\end{lem}
\begin{proof}
	For the moment, we take $G_n$ to be the one dimensional torus $\dT_n$. To distinguish this from the usual non-transitive case, we denote the partition function as $W_n$. It is equivalent to prove a lower bound for $\var\log W_n$. This is because a modification of~\ref{def:Decomp} yields that $\log Z_{n-1} + R_{n+1,1} = \log W_{n+1}$, where $\norm{R}_{2}$ is bounded above. On centering,
	\begin{align*}
		\norm{\ol{\log Z_{n-1}}}_{2}+\norm{\ol{R_{n+1,1}}}_{2}\ge \norm{\ol{\log W_{n+1}}}_{2}.
	\end{align*}
	Consider the principal subgraph obtained on restriction to the vertices having fixed $H$--coordinate, let us say $k\in V_{H}$. There are $n$ edges corresponding to this layer, and we enumerate them as $\{e_1,e_{2},\ldots, e_n\}$. We denote their respective weights as $\{\go_1,\go_{2},\ldots,\go_n\}$. Let
	\begin{align*}
		\cF_j=\gs\{\go_1,\go_{2},\ldots, \go_j\}
	\end{align*}
	denote the filtration corresponding to the weights of the edges.
	We will use Doob's martingale decomposition to find a lower bound for the variance.

	We have that $\E\left(\log W_n\mid \cF_j\right)$ is a martingale in $j$, and the variance of $\E\left(\log W_n \mid \cF_n\right)$ can be written as the sum of the variances of the martingale differences. The $j^{th}$ martingale difference is
	\begin{align}\label{def:martdiff}
		\E\left(\log W_n\mid \cF_j\right)-\E\left(\log W_n\mid\cF_{j-1}\right).
	\end{align}
	Suppose we replace the weight $\go_j$ of edge $e_j$ with an independent copy $\go_j'$, let $W_n^{(j)}$ denote the new partition function. The expression in~\eqref{def:martdiff} is the same as
	\begin{align*}
		\E\bigl(\log W_n-\log W_n^{(j)}\mid \cF_j\bigr).
	\end{align*}
	Thus,
	\begin{align*}
		\var(\log W_n)
		\ge \var(\E\left(W_n\mid \cF_n\right))
		= \sum_{j=1}^n \norm{\E(\log W_n-\log W_n^{(j)}\mid\cF_j)}^2_{2}.
	\end{align*}
	A combination of the tower property and Jensen's inequality yields
	\begin{align}\label{def:varlb}
		\var(\E\left(\log W_n\mid \cF_n\right))
		\ge \sum_{j=1}^n\norm{ \E(\log W_n -\log W_n^{(j)}\mid\go_j)}^2_{2}
		= \sum_{j=1}^n\var(\E\left(\log W_n \mid\go_j\right)).
	\end{align}
	Shift-invariance of $G_n$ implies edge transitivity; therefore, exchangeability of the random variables we are conditioning on. Concretely, the random variables 
	$
		\E\left(\log W_n\mid \go_i\right)
	$ depends on $i$ only through $\go_i$, for all $i=1,2,\ldots,n$. We write
	\[
	 g_n\left(\go_1\right):=\E\left(\log W_n\mid \go_1\right)
	\]
	as the common conditional expectation function.
	
	The problem of the variance lower bound for the free energy has been reduced to showing that the variance of $g$ is bounded below by a constant, since from~\eqref{def:varlb} we have
	\begin{align}\label{def:reducetog}
		\var(\log W_n)
		\ge n\var(g_n(\go_1)).
	\end{align}
	Recall the $\ga_{e_j}$ and $\gb_{e_j}$ introduced in Lemma~\ref{lem:efron}, so that $W_n=\ga_{e_j} + \gb_{e_j} e^{\go_j}$ and 
	\begin{align}\label{def:perturb}
		\log W_n -\log W_n^{(j)}=\int_{\go_j'}^{\go_j}\frac{\gb_{e_j} e^z}{\ga_{e_j} +\gb_{e_j} e^z}\, dz
	\end{align}
	for all $j=1,2,\ldots,n$. W.l.o.g.~we can work with $j=1$.
	Let $u$ and $v$ be the vertices incident to $e_1$, and let $\cE_1\subset E_{\cG_n}$ denote the collection of edges adjacent to $e_1$. From the definitions of $\ga_{e_1}$ and $\gb_{e_1}$, on using~\eqref{def:recur} we have
	\begin{align*}
	 e^{\nu_u+\nu_v} \cdot \gb_{e_1} 
	 \le \ga_{e_1} 
	 \le
	 e^{\nu_u+\nu_v} \prod_{e=(x,y)\in \cE_1} (1+e^{\go_e-\nu_x-\nu_y})
	 \cdot \gb_{e_1}.
	\end{align*}
	Note that, the random variable 
	\[
	 U:=1/(1+e^{\nu_u+\nu_v} \prod_{e=(x,y)\in \cE_1} (1+e^{\go_e-\nu_x-\nu_y}))
	\]
	is bounded and has strictly positive mean as $\abs{\cE_1}$ is bounded. 
	Observe that 
	\begin{align*}
	 \var( g_n(\go_1))
	 =\frac{1}{2}\E\left(g_n(\go_1)-g_n(\go_1')\right)^{2}
	 =\E\left(\1_{\go_1
	 > \go_1'}
	 \left(g_n(\go_1)-g_n(\go_1')\right)^{2}\right).
	\end{align*}
	Since the integrand in~\eqref{def:perturb} is monotonically increasing in $z$, on the event $\{\go_1> \go_1'\}$, we have
	\begin{align*}
	 \log W_n-\log W_n^{(1)}
	 \ge \frac{\gb_{e_1}}{\ga_{e_1} + \gb_{e_1}} (\go_1-\go_1')_{+} 
	 \ge (\go_1-\go^{'}_1)_{+} \cdot U.
	\end{align*}
	Thus,
	\begin{align*}
		\left(g_n(\go_1)-g_n(\go_1') \right)\1_{\go_1> \go_1'}
		=\1_{\go_1> \go_1'}\E\left(\log W_n-\log W_n^1\mid \go_1,\go_1'\right)
		\ge (\go_1-\go_1^{'})_{+} \E(U).
	\end{align*}
	To complete the proof, we square, take the expectation, and apply~\eqref{def:reducetog}.
\end{proof}

\subsubsection{Proof of Theorem~\ref{thm:Zmeanvar}}
\label{ssec:pfofZmeanvar}
Now we have all the ingredients to prove Theorem~\ref{thm:Zmeanvar}. Convergence of the free energy follows from Lemma~\ref{lem:freeenergy_mean} and Corollary~\ref{cor:Zmean}. The variance result follows from Lemmas~\ref{lem:zvarconv} and~\ref{lem:zvarlowerbound}.
\qed

\subsection{Central Limit Theorem for the Free Energy}
Having shown that the limiting variance exists and is non-degenerate, we are ready to prove a Gaussian central limit theorem for the partition function. We will repeatedly use the decomposition stated in~\eqref{def:Decomp} to express the partition function as an approximate sum of i.i.d.~random variables. We will dyadically cut $\cG_n$ into disjoint blocks of equal size by dropping layers of edges corresponding to edges of $G_n$. The partition functions associated with these disjoint blocks are independent, and we will repeat this process to approximate the free energy $\log Z_n$ by a sum of i.i.d.~random variables. Heuristically, if the length of a block is even, we remove the central layer to obtain two disjoint blocks incurring an error $R$. If the length is odd, we drop the terminal layer of vertices incurring an error $T$ and proceed with the central cut. Note that, after we subdivide $2$ disjoint blocks into $4$, the errors $R_1$ and $R_{2}$ corresponding to cutting each of the original blocks are i.i.d. We now formally describe the subdivision. Consider the dyadic expansion of $n$, \ie
\begin{align*}
	n = \sum_{i=0}^\ell a_i\cdot 2^i \text{ with }a_i=\lfloor n/2^i \rfloor.
\end{align*}

We define two useful operations on $\cX=\{\mva\in\{0,1\}^\infty\mid \mva \text{ has finitely many $1$ entries}\}$. The evaluation map $\pi:\cX\to \dN$ takes the tuple to the associated natural number for which it is the dyadic expansion, \ie
\begin{align*}
	\pi(\mva)=\sum_{i=0}^{\infty}a_i\cdot 2^i
\end{align*}
and the left-shift map $s:\cX\to \cX$ is defined by dropping the first element of the tuple, \ie
\begin{align*}
	s(a_0,a_1,a_2,\ldots) =(a_1,a_2,\ldots). \end{align*}
We denote the $k^{\text{th}}$ iterate of the left-shift map as $s^k$, and denote
\begin{align*}
	\pi_k(\mva):=\pi \circ s^k(\mva)\text{ for } \mva\in\cX.
\end{align*}
The scheme for subdivision is defined for one iteration and extended inductively. Let $\mv{a}$ be the dyadic expansion of $n$. Suppose $a_{0}=0$. The subdivision can take place without the loss of the terminal portion of the block to obtain two blocks of length $\pi_1(\mv{a})$ and an error $R$. If $a_{0}=1$, we drop the terminal layer of vertices incurring an error $T$ and subdivide the block that remains again, obtaining disjoint blocks of size $\pi_1(\mv{a})$. In the first subdivision, we have the following
\begin{align*}
	\log Z_n=\log Z^1_{\pi_1(\mv{a})} + \log Z^{2}_{\pi_1(\mv{a})} + a_{0}T_{0}+R_{0}.
\end{align*}
On the $k^{th}$ subdivison we have disjoint blocks of size $\pi_{k}(\mva)$ as well as corresponding error terms, both of which can be indexed by the binary tree. We will denote the vertices in the $k^{th}$ generation as $V_{k}$. For a given generation $k$ and $v\in V_{k}$, $Z_{\pi_{k}(\mva)}^{v}$ denotes an independent copy of $Z_{\pi_{k}(\mva)}$, $R_{k}^{v}$ the central error arising in subdividing $Z^{v}_{\pi_{k}(\mva)}$ into two copies of $Z_{\pi_{k+1}(\mv{a})}$, and $T_{k}^{v}$ the terminal error associated to the same subdivision. It is clear that for fixed $k$, the $R_{k}^{v}$ are i.i.d., the same for $T_{k}^{v}$. 

In particular, we have
\begin{align}\label{def:dyadic}
	\log Z_n=\sum_{v\in V_{k}} \log Z^{v}_{\pi_{k}(\mva)}+\sum_{j=1}^k\sum_{v\in V_{j-1}} \bigl( R_{j-1}^{v}+a_{j-1}T_{j-1}^{v}\bigr).
\end{align}
The first step of subdivision is illustrated in Figure~\ref{fig:subdivide} for the case where $a_{0}=1$. The terminal layer of vertices and the central layer of edges to remove are highlighted. By construction, the family $\{\{Z^{v}_{\pi_{k}(\mva)}\}_{v\in V_{k}}\}_{k\le l}$ forms a triangular array. We will now use the Rosenthal bound in Theorem~\ref{thm:rosenthal} to show that we have the moment control required to apply the Lyapunov Central Limit Theorem.
\begin{figure}[htbp]
	\includegraphics[width=4in,page=1]{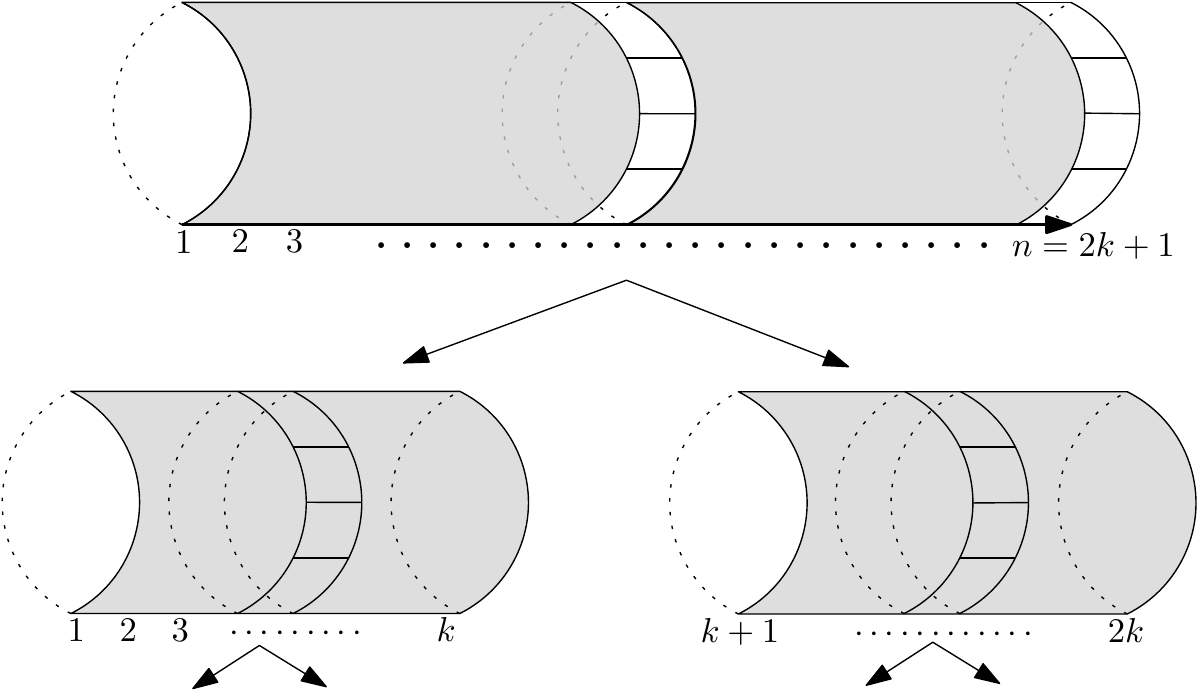}
	\caption{ First step of the subdivision}
	\label{fig:subdivide}
\end{figure}

\begin{lem}\label{lem:lyapunov}
	Let $\eps >0$ be fixed and $\norm{\go}_{2+\eps}+\norm{\nu}_{2+\eps}<\infty$. There exists a fixed constant $C$ such that that
	\begin{align*}
		\norm{\ol{\log Z_n}}_{2+\eps}\le C\sqrt{n}.
	\end{align*}
\end{lem}

\begin{proof}
	To prove this lemma, we will have to subdivide until we have blocks of constant order size. We begin with~\eqref{def:dyadic} and take $k=\ell -2$. On centering and applying the triangle inequality,
	\begin{align*}
		\norm{\ol{\log Z}_n}_{2+\eps}\le \norm{\sum_{v\in V_{\ell}}\ol{\log Z^{v}}_{\pi_{l-2}(\mv{a})}}_{2+\eps}+\sum_{j=1}^{\ell-2}\norm{\sum_{v\in V_{j-1}}\ol{R^{v}_{j-1}}}_{2+\eps}+\sum_{j=1}^{\ell-2}a_{j-1}\cdot \norm{\sum_{v\in V_{j-1}}\ol{T^{v}_{j-1}}}_{2+\eps}.
	\end{align*}
	Note that each of the sums within the norms are sums of i.i.d.~mean zero random variables, justifying the application of Rosenthal's inequality. Lemma~\ref{lem:errbound} yields uniform bounds for $\norm{T_j}_{2+\eps}$ and $\norm{R_j}_{2+\eps}$ and the proof of Theorem~\ref{thm:Zmeanvar} yields a bound on $\norm{Z_{\pi_{\ell-2}(\mv{a})}}$. Applying Theorem~\ref{thm:rosenthal},
	\begin{align*}
		\norm{\ol{\log Z}_{(a_{\ell}\ldots a_{0})}}_{2+\eps}
		\le 2^{\frac{\ell-2}{2}}C_1+C_{2}\sum_{j=1}^{l-2}(1+a_j)2^{\frac{j}{2}}.
	\end{align*}
	where $C_1$ and $C_{2}$ are constants determined by $\norm{\go}_{2+\eps}$, $\norm{\nu}_{2+\eps}$ and $\eps$. Clearly, we have $1+a_j\le 2$ and by hypothesis, $2^{\ell}\le n\le 2^{\ell+1}$, implying that
	\begin{align*}
		n^{-\half}\norm{\ol{\log Z}_n}_{2+\eps}\le C_1+2^{\frac{3}{2}}C_{2}.
	\end{align*}
	Taking $C$ to be the right hand side constant completes the proof.
\end{proof}

The $(2+\eps)$-th moment condition is required for the Lyapunov condition to hold. We are now ready to prove Theorem~\ref{thm:Zclt}.

\subsubsection{Proof of Theorem~\ref{thm:Zclt}}
	Let $n$ have the dyadic representation given by $(a_{0},a_1,\ldots a_{\ell})$, and let us carry out the subdivison in~\eqref{def:dyadic} to depth $k=\lfloor\ell/2 \rfloor$ yielding blocks of size $m$ where $2^{\frac{\ell}{2}}\le m\le 2^{\frac{\ell+1}{2}}$
	\begin{align*}
		n^{-\half}\ol{\log Z_n}
		=n^{-\half}\sum_{v\in V_{k}}
		\ol{\log Z_{\pi_{k}(\mva)}^{v}}+n^{-\half}\sum_{j=1}^k\sum_{v\in V_{j-1}} \bigl(\ol{R_{j-1}^{v}}+a_{j-1}\ol{T_{j-1}^{v}}\bigr).
	\end{align*}
	It is easy to show that the error terms vanish in the limit, by the same argument used in Lemma~\ref{lem:lyapunov} we have that
	\begin{align*}
		n^{-\half}\sum_{j=1}^k\sum_{v\in V_j} \norm{\ol{R_{j-1}^{v}}+a_{j-1}\ol{T_{j-1}^{v}}}_{2}\le 2^{-k}\cdot C2^{ \frac{k}{2}+1} \to 0.
	\end{align*}
	By Slutzky's theorem, we need only worry about the distributional convergence of
	\begin{align*}
		n^{-\half}\sum_{v\in V_{k}} \ol{\log Z^{v}_{\pi_{k}(\mva)}}.
	\end{align*}
	The explicit relation between $n$ and $\ell$ is given by $\ell=\lfloor {\log n}/{\log 2} \rfloor $, and it is clear that $m=\pi_{k}(\mva)=\lfloor {n}/{2^k}\rfloor$. From these definitions of $\ell$, $k$ and $m$ in terms of $n$, it is also clear that $\sqrt{{n}/({m2^k})} \to 1$. Consider the triangular array defined by $\ol{\log Z^{v}_{m}}$, where $v\in V_{\lfloor {\ell}/{2}\rfloor}$. For fixed $m$, they are i.i.d. as $v$ varies. Lemma~\ref{lem:lyapunov} verifies that the hypotheses of the Lyapunov Central Limit Theorem hold and we may directly apply it to obtain
	\begin{align*}
		{\bigl({m2^{\lfloor\ell/2\rfloor}}\bigr)^{-\half}}\cdot \sum_{v\in V_{\lfloor\frac{\ell}{2}\rfloor}} \ol{\log Z^{v}_{m}} \to \N(0,\sigma_{F}^2),
	\end{align*}
	as desired.\qed
	
\subsection{Ground State Energy}

Here we discuss the asymptotic behavior of the Ground State energy or the free energy in the setting of zero temperature. One can hope to extend the techniques used in the finite temperature case to the zero temperature case, in order to obtain the scaling and distributional limit of the ground state energy. In our model, this mainly requires two ingredients: the error decomposition as in~\eqref{def:Decomp} and a variance control. Note that, given a matching $\fm$ on $\cG_{n}$, we may restrict it to the sections $\cG_{[1:k]}$ and $\cG_{[k+1:n]}$, and obtain matchings denoted by $\fm_{[1:k]}$ and $\fm_{[k+1:n]}$, respectively. We will use $M_{[1:k]}$ and $M_{[k+1:n]}$ to denote the ground state energies on sections $\cG_{[1:k]}$ and $\cG_{[k+1:n]}$, respectively. 
Similar to the decomposition given in~\eqref{def:Decomp}, we define the ``error'' random variable
\[
\sR_{n,k} := M_n - (M_{[1:k]}+M_{[k+1:n]}).
\] 
It is trivial to show that
 \begin{align}
 \sR_{n,k}\ge 0.
 \end{align} 
 
Now, the restriction of any matching $\fm$ to the sections $\cG_{[1:k]}$ and $\cG_{[k+1:n]}$, removes all edges present in the bridging layer $\cE_{k}$. In particular, using any ground state $\fm^\star$, \ie\ a matching satisfying $\cH(\fm^\star)=M_n$, we have
\begin{align*}
 \sR_{n,k} 
 &\le \cH(\fm^{\star}) - \cH(\fm^{\star}_{[1:k]}) - \cH(\fm^{\star}_{[k+1:n]}) \\
 &= \sum_{i:e_{k,i}\in \fm^{\star}} (\go_{k,i}-\nu_{k,i}-\nu_{k+1,i})
 \le \sum_{i} (\go_{k,i}-\nu_{k,i}-\nu_{k+1,i})_+.
\end{align*}
Thus we have the following lemma.

 \begin{lem}
Let $\norm{\go_{e}}_{p}+\norm{\nu_{v}}_{p}<\infty$ for some $p\ge 1$. There exists a constant $C$ depending only on $p$ such that \[
\norm{\sR_{n,k}}_p\le C(\norm{\go_{e}}_{p}+\norm{\nu_{v}}_{p}).
\]
 \end{lem}

Using the subadditive argument from Theorem~\ref{thm:subadditive}, we have the following corollary.

\begin{cor}
 Let $\norm{\go_{e}}_{1}+\norm{\nu_{v}}_{1}<\infty$, then there exists $m\in \dR$ such that,
 \[
 n^{-1}\E M_{n} \to m \text{ as }n\to \infty. 
 \]
\end{cor}

Moving to the question of the fluctuations, we establish a variance upper bound using the Efron-Stein inequality. 
\begin{lem}
There exists a constant $C\in (0,\infty)$ such that 
\[
\var{M_{n}}\le C n \text{ for all } n
\]
\end{lem}
\begin{proof}
We illustrate the method for the edge weights, as the case of bounding the change with respect to the vertex weights is identical. Let us pick an edge $e \in E_{\cG_{n}}$ with weight $\go_{e}$. We replace $\go_{e}$ with an independent copy $\go_{e}'$. Note that, 
\[ \E(M_{n}-M_{n}^{e})^2=2\E(\1_{\go_e> \go'_e}\cdot(M_{n}-M_{n}^{e})^2)
\]
and it is easy to see that 
\[
0\le \1_{\go_e> \go'_e}\cdot(M_{n}-M_{n}^{e}) \le (\go_e-\go'_e)_+.
\]
Adding over all edges $e$, then repeating the same procedure with the vertex weights completes the proof. 
\end{proof}

We thus have two immediate corollaries.
\begin{cor}\label{cor:Mconv}
 Let $\norm{\go_{e}}_{2}+\norm{\nu_{v}}_{2}<\infty$. Then we have 
 $
 n^{-1} M_{n}\toP m \text{ as } n\to \infty. 
 $
\end{cor}

\begin{cor}\label{cor:Mvarconv}
 We have $\gs_{M}\in[0,\infty)$ such that 
 \[
 n^{-1}\var M_{n} \to \gs_{M}^2 \text{ as }n\to \infty. 
 \]
\end{cor}
\begin{proof}
Identical procedure to that of Lemma~\ref{lem:zvarconv}. 
\end{proof}

The dyadic subdivision introduced in the proof of Theorem~\ref{thm:Zclt} is easily carried out in this context. Lemma~\ref{thm:rosenthal} may be applied to show that $\norm{\ol{M_{n}}}_{2+\epsilon}\le Cn^{\half}$ whenever the weights satisfy the appropriate moment condition. The Lyapunov condition to prove a central limit theorem holds.

\subsubsection{Proof of Theorem~\ref{thm:gse}}
The proof of the convergence of $n^{-1}M_{n}$ follows from Corollary \ref{cor:Mconv}, the convergence of the variance to $\gs_{M}$ from \ref{cor:Mvarconv}. The $(2+\epsilon)^{\th}$ moment bound follows from a procedure identical to Lemma~\ref{lem:lyapunov}, and the proof of the Central Limit Theorem is identical to Theorem~\ref{thm:Zclt}. We note that we have not established a variance lower bound, if $\gs_{M}=0$, it means that $\|n^{-\half}\cdot (M_{n}-\E M_{n})\|_{2}\to 0$. If on the other hand $\gs_{M}>0$, we have convergence in distribution to $\N(0,\gs_{M}^{2})$
\qed

%%%%%%%%%%%%%%%%%%%%%%%%%%%%%%%%%%%%%%%%%%
\section{Exponential Tilting and Lee-Yang Zeroes} \label{sec:LeeYang}
%%%%%%%%%%%%%%%%%%%%%%%%%%%%%%%%%%%%%%%%%%%
\subsection{Gauge Transformation and Exponential Tilting} We move towards addressing the question of typical behavior of a matching $\fm$ chosen according to $\mu$. We will do this by characterizing the behavior of $U$, the number of unpaired vertices. There is a standard technique in statistical mechanics for computing cumulants of observables via the exponential tilting of the Gibbs measure. Since the observable we are interested in is $U$, let $x\in \dR$ and define the modified Hamiltonian by
\begin{align*}
	\cH_x(\fm):=xU(\fm)+\cH(\fm).
\end{align*}
We then have the modified partition function
\begin{align*}Z_n(x):=\sum_{\fm \in \cM} \exp(\cH_x(\fm))\end{align*}
and finally the tilted measure
\begin{align*}
	\mu_x(\fm):= \frac{\exp(\cH_x(\fm))}{Z_n(x)}.
\end{align*}
Suppose we have a random variable $X:\cM \to \dR$, then the Gibbs average with respect to the tilted measure will be denoted as $\la X\ra_n^x$.
We recover the original measure when $x=0$, and clearly $\la X\ra^{0}_n=\la X\ra_n$. The derivatives of $\log Z_n(x)$ with respect to $x$ yield the cumulants of $U$. In particular, for the first two cumulants, we have
\begin{align}
	\la U \ra^x_n &=\partial_x \log Z_n(x),\label{def:umean}\\
\text{ and }
	\left\la U^{2} \right\ra^x_n-\bigl(\la U\ra_n^x\bigr)^{2} &=\partial_x^{2} \log Z_n(x).\label{def:uvar}
\end{align}
It is equivalent to regard the exponential tilting by $xU(\fm)$ as re-weighting the vertices, for all $u\in V$ $\nu_u$ becomes $\nu_u+x$. In this section, the independence of the edge weights will not be as important, the vertex weights being constant is crucial. By the gauge invariance of the model we may pass the random weights of the vertices onto the edges by successively applying the following gauge transformation at each vertex $u$:
\begin{align*}
	x+\nu_u\to x \text{ and } \go_{(u,v)} \to \go_{(u,v)}-\nu_v \text{ for all } v\sim u.
\end{align*}
The transformed weights on the edges are thus
\begin{align*}
	\tilde{\go}_{(u,v)}=\go_{(u,v)}-\nu_v-\nu_u.
\end{align*}

\subsection{Interlacing and the Empirical Measure} The partition function $Z_n(x)$ is a polynomial in $e^x$. The gauge transformation, in essence is to make this polynomial monic. We define the Lee-Yang zeroes in this context. Let $\tilde{Z}$ denote the gauge transformed partition function which is a monic polynomial of order $N=nh$ in $e^x$. The partition function is real and positive, implying the roots have to exist in conjugate pairs. It is easy to see that for a graph on two vertices, the roots are purely imaginary. Combining~\eqref{def:recur} with an induction argument yields that the roots of $\tilde{Z}_n$ are purely imaginary as well. Thus, we have
\begin{align}
	\tilde{Z}_n(x)=\prod_{i=1}^{N}(e^x+\gl_i\sqrt{-1}).
\end{align}
As an abuse of terminology, moving forward when we refer to the Lee-Yang zeroes, we will be referring to the collection $\{\gl_i\}_{i=1}^{N}$, in non-decreasing order. We recall some results on Lee-Yang zeroes from~\cite{HL}, as mentioned in the introduction. Let $\cG$ denote a weighted graph on $n$ vertices with edge weights $\{\tilde{\go}_e\}_{e\in E_{\cG}}$ and constant vertex weights $x$. Let $u$, $v \in V_{\cG}$ be vertices, and $\cG^{u}$, $\cG^{u,v}$ be the principal subgraphs obtained on the removal of $u$, and both $u$ and $v$ respectively. Then we have the following recurrence relation on the corresponding partition functions
\begin{align}\label{def:recur2}
	Z_{\cG}(x)=\exp(x)\cdot Z_{\cG^{u}}(x)+\sum_{v\sim u}\exp(\tilde{\go}_{u,v})\cdot Z_{\cG^{u,v}}(x).
\end{align}
Let $\{v_1,v_{2},\ldots ,v_{k}\}\subset V_{\cG}$ be a collection of vertices. We sequentially remove them to obtain a sequence of principal subgraphs denoted by $\{\cG^{(j)}\}_{i=0}^k$ with corresponding partition functions $Z^{(j)}(x)$. The recurrence~\eqref{def:recur2} implies an interlacing hierarchy. Let the Lee-Yang zeroes of $\cG^{(j)}$ will be denoted by $\{\gl_{k,j} \}_{k=1}^{n-j}$. We have
\begin{align}\label{def:interlace}
	\gl_{k,j+1}\le \gl_{k,j}\le \gl_{k+1,j+1}.
\end{align}
The interlacing heirarchy is illustrated for the first $5$ layers in Figure~\ref{fig:interlacing}. Also illustrated is the ``cone of comparision'', the directions along which we may use the zeroes in lower levels to obtain lower or upper bounds on the zeroes in the upper levels. This notion will be of prime importance to us. Heilmann and Lieb also proved a localization result for the zeroes in~\cite{HL}, which we state in the form used here. Let $\Delta_u$ denote the weighted degree of a vertex $v$, that is
\begin{align}\label{weightedDegree}
	\Delta_u=\sum_{e\sim u}\exp(\tilde{\go}_e).
\end{align}
Then for all $1\le i\le n$, we have
\begin{align}\label{def:zerobound}
	|\gl_i|\le\sup_{u\in V} \Delta_u.
\end{align}
We note that the distribution of $\Delta_u$ is determined by the degree of $u$, that is if $\text{deg}(u)=\text{deg}(v)$ for vertices $u$ and $v$ then $\Delta_u\equald\Delta_v$. This is because we may write \begin{align*}\Delta_u=\exp(-\nu_u)\sum_{v\sim u}\exp(\go_{(u,v)}-\nu_v)\end{align*}
where all the random variables appearing on the right are independent of each other.

\begin{figure}\label{fig:interlacing}
	\includegraphics[scale=0.6,page=2]{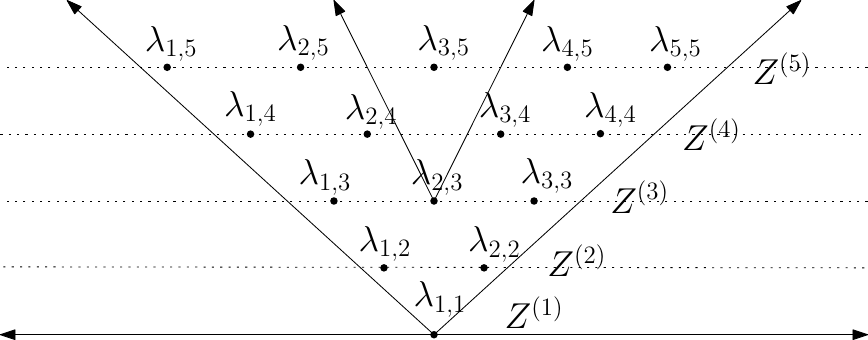}
	\caption{Interlacing shown for the first 5 levels}
\end{figure}

To the roots $\{\gl_i\}_{i=1}^{N}$ we associate the empirical distribution $\rho_n$, a probability measure on the real line given by
\begin{align*}\rho_n=n^{-1}\sum \delta_{\gl_j}.\end{align*} Here, $\delta_x$ denotes the Dirac measure with unit mass centered at $x$. The convergence of relevant thermodynamic quantities such as the cumulants of $U$ can be phrased in terms of the weak convergence of this sequence of empirical counting measures. We have
\begin{align*}
	n^{-1}\log \tilde{Z}_n=\frac{1}{2}\int_{\dR} \log (\gl^2+e^{2x}) d\rho_n(\gl).
\end{align*}
We may rewrite~\eqref{def:umean} and~\eqref{def:uvar} in terms of the empirical distribution of the zeroes as
\begin{align*}
	n^{-1}\la U \ra^x _n & = \int_{\dR} (1+\gl^2e^{-2x})^{-1} d\rho_n(\gl), \\
	\text{ and }
	n^{-1}\left(\left\la U^{2} \right\ra_n-\la U\ra^{2}_n\right)
	 & = \int_{\dR} 2\gl^2e^{-2x}(1+\gl^2e^{-2x})^{-2} d\rho_n(\gl).
\end{align*}
While we may define the tilted measure for all $x \in \dR$, for our purposes it suffices to take $x$ in some compact interval.

\begin{lem}\label{lem:tightness}
	Let $\cG_n$ have maximal degree $d_{max}<\infty$. Then the sequence of empirical measures $\{\rho_n\}$ is tight in probability.
\end{lem}

\begin{proof}
	We begin by constructing a sequence of principal subgraphs of $\cG_n$ whose corresponding Lee-Yang zeroes satisfy the interlacing condition in~\eqref{def:interlace}.
	Let $K \in (0,\infty)$ denote a cutoff. Recall the definition of $\Delta_u$ from~\eqref{weightedDegree}, the weighted degree of a vertex $u$. We remove all vertices $u$ with $\Delta_u\ge K$ from $\cG_n$, and denote the principal subgraph on the remaining vertices as $\cG^{K}$. Clearly, as $K$ increases, the vertices are added in one by one until we obtain the full graph $\cG$ as $K \to \infty$. Let $N_K$ denote the number of vertices in $\cG^{K}$. Clearly, \begin{align*}N_K=\sum_{v\in V}\1_{\Delta_v<K}.\end{align*}
	Now, let $u$ and $v$ be two distinct vertices. Via a coupling argument, it is not hard to show that if $\text{deg}(u)>\text{deg}(v)$, then $\pr(\Delta_u<K)\le \pr(\Delta_v<K)$. Let $\Delta_{max}$ denote a random variable equal in distribution to the weighted degree of the vertex with maximum degree. We find the following bounds for the mean and variance

	\begin{align*}N\pr\left(\Delta_{max}< K \right)\le \E N_K \le N \pr\left(\tilde{\go}_e< \log K \right).\end{align*}
	Since $N=nh$ and $\abs{ E_{\cG_n} }\le (d_{max}+1)nh/2$,

	\begin{align*}
		\var(N_K)\le \sum_{v \in V} \var(\1_{\Delta_v< K}) + \sum_{v\sim w} \cov(\1_{\Delta_v< K},\1_{\Delta_{w}< K}) \le (d_{max}+3)nh/2.\end{align*}
	These estimates together yield concentration via Chebyshev's inequality:
	\begin{align*} 
	\pr \left(\abs{ N_K-\E N_K }\ge \eps n\right)\le \frac{(d_{max}+3)h}{2\eps^2 n}.
	\end{align*}
	What this means in particular is that for $\eps >0$ with probability approaching $1$, \begin{align*}N_K/n\ge \pr\left(\Delta_{max}<K\right)-\eps.\end{align*}
	We do not go into the detail of the distribution of $\Delta_{max}$. This is because for proving tightness of $\rho_n$ all that is required is that $\Delta_{max}$ is a tight random variable which is clear since $d_{max}$ is bounded. The localization of the zeroes relates the weighted degree of the vertices to the absolute values of the Lee-Yang zeroes. By hypothesis, on $\cG^{K}$ the weighted degrees of all the vertices are bounded above by $K$, thus all Lee-Yang zeroes corresponding to $Z_{\cG^{K}}$ are located in the interval $[-K,K]$. Let us now take $K$ and $K+\delta$ such that exactly one vertex is added, that is $N_{K+\delta}=N_K+1$. Let the Lee-Yang zeroes of the corresponding partition functions be denoted as $\{\gl_{k,K+\delta}\}_{k=1}^{N_{K+\delta}}$ and $\{\gl_{k,K}\}_{k=1}^{N_K}$. The interlacing~\eqref{def:interlace} yields that
	\begin{align*}\gl_{N_K,K}\ge \gl_{N_K,K+\delta}.\end{align*} Extending this inductively, let $\{\gl_{k}\}^{N}_{k=1}$ denote the Lee-Yang zeroes of the partition function corresponding to $\cG_n$. We have
	\begin{align*}
		\gl_{N_K}\le \gl_{N_K,K}\le K.
	\end{align*}

	Thus, $2(n-N_K)/n$ is an upper bound for the fraction of Lee-Yang zeroes of $Z_{\cG_n}(x)$ such that $|\gl|> K$. Given an $\eps > 0$, we can choose $K$ such that $2(n-N_K)/n\le \eps$ with probability approaching 1. This is exactly $\rho_n[-K,K]^{c}$, which establishes tightness.
\end{proof}

The tightness condition allows us to rule out the possibility of the mass of $\rho_{n}$ leaking away to infinity. We may also prove that it is impossible for the mass of $\rho_{n}$ to be entirely concentrated on $0$.

\begin{lem}\label{lem:zeroanticonc}
There exists $\eps,\gd>0$ such that with probability approaching $1$, 
\[
\rho_{n}([\gd,\infty))\ge \eps \text{ for all } n.
\]
\end{lem}
\begin{proof}
We focus on a particular principal subgraph of $\cG_n$, which consists of disjoint connected components of size $2$. All vertices $v$ in $\cG_n$ are of the form $v=(g,h)$ where $g\in V_{G}$ and $h\in V_{H}$ are the $G$ and $H$ coordinates, respectively. 
	We fix a $H$-coordinate, say $h^\star\in V_{H}$. Let 
	\[
	 \cV'=\{((i,h^\star),(i+1,h^\star)): i=1\text{ mod }2\}
	\]
	be a collection of $\lfloor n/2\rfloor$ many disjoint edges. 
	We obtain a sequence of principal subgraphs of $\cG_n$ by sequentially removing all vertices not in $\cV'$, until only the vertices present in $\cV'$ are left. We denote the resulting subgraph as $\cG_{\cV'}$, and the size of its vertex set as $N_{\cV'}$. By construction, 
	\begin{align*}
	 N_{\cV'}= 2\lfloor n/2 \rfloor.
	\end{align*}
	Note that each connected component of $\cG_{\cV'}$ is of size $2$, that is just a pair of vertices connected by an edge. Consider $K>0$ such that $p:=\pr(\tilde{\go}_e \ge -K)>0$. Define $\gd:=\exp(-K/2)$. We sequentially remove the vertices adjacent to edges with edge weight $\tilde{\go}<-K$. We denote the resulting subgraph as $\cG_{\cV}$ and the size of its vertex set as $N_{\cV}$. As a standard consequence of the concentration of Binomial random variables, for every $\eps>0$ we have with probability approaching $1$, we have
	\[
		N_{\cV}>N_{\cV'}\cdot (p-\eps).
	\]
	The connected components of $\cG_{\cV}$ have size 2. Thus, the partition function $Z_{\cG_{\cV}}$ factors into a product of the partition functions of the respective connected components. We denote the Lee-Yang zeroes of $Z_{\cG_{\cV}}$ as $\{\gl_{i,\cV}\}_{i=1}^{N_{\cV}}$. Now, if $G'$ is a graph on two vertices with edge weight $\tilde{\go}$ and vertex weights $x$, it is easy to see that the partition function $Z_{G'}$ must be of the form
	\begin{align*}
		Z_{G'}(x)=\exp({2x})+\exp({\tilde{\go}})=\left(\exp({x})+\sqrt{-\exp(\tilde{\go})}\right)\left(\exp(x)-\sqrt{-\exp(\tilde{\go})}\right)
	\end{align*}
	with Lee-Yang zeroes given by $\pm \exp(\tilde{\go}/2)$. 
	Given the product structure of $Z_{\cG_{\cV}}$, it follows that $|\gl_{i,\cV}|\ge e^{-K/2}=\gd$ for all $1\le i\le N_{\cV}$. Let $\ell=N_{\cV}/2$.

	Via the interlacing hierarchy, we obtain that
	\begin{align*}
		\gl_{\ell+N-\cN_{\cV}}\ge \gl_{\ell,\cV}\ge \gd
	\end{align*}
	where $\{\gl_i\}_{i=1}^{N}$ are the Lee-Yang zeroes of $Z_n(x)$. Further, this means that there are at least $N_{\cV}-\ell$ zeroes that exceed $\gd$. With probability approaching $1$, we have
	\begin{align*}
		\rho_n([\gd,\infty))\ge \frac{N_{\cV}}{2N}\ge \frac{2\lfloor n/2\rfloor (p-p/2)}{2nh}>\frac{p}{4h}-\frac{p}{2nh}.
	\end{align*}
	To conclude, we take $\eps=p/(4h)$. 
\end{proof}

%%%%%%%%%%%%%%%%%%%%%%%%%%%%%%%%%%%%%%%
\subsection{Cumulant Comparison}
%%%%%%%%%%%%%%%%%%%%%%%%%%%%%%%%%%%%%%%

In addition to establishing tightness, interlacing establishes a bound on the influence of removal of a vertex on the cumulants of $U$.
\begin{lem}\label{lem:vertexremovalbound}
	Let $\cG$ be any graph with weighted edges and vertex weights given by $x$, and let $N=|V_{\cG}|$. Let $V'\subset V_{\cG}$ denote a collection of vertices. Let $\cG^{V'}$ denote the principal subgraph obtained on removal of all vertices in $V'$, and let $Z^{V'}(x)$ denote the corresponding partition function. We have a finite constant $C$ depending on $x$ and $i$ such that
	\begin{align*}\abs{ \partial_x^i\log Z(x)-\partial_x^i\log Z^{V'}(x) }\le C|V'|.\end{align*}
\end{lem}

\begin{proof}
	Let us enumerate the vertices in $V'$ as $\{v_1,v_{2},\ldots, v_{|V'|}\}$. We then sequentially remove the vertices, obtaining a sequence of partition functions whose underlying graphs differ by the removal of a single vertex. We denote the partition function after the $j^{\th}$ removal by $Z^{(j)}(x)$ for all $1\le j\le |V'|$. Clearly, $Z^{(|V'|)}(x):=Z^{V'}(x)$. The Lee-Yang zeroes of $Z^{(j)}(x)$, denoted $\gl_{k,j}$, interlace those of $Z^{(j-1)}(x)$. For ease of notation, we denote
	\begin{align}\label{def:fidef}
		f^i(x,\gl):=\partial_x^{i-1}(1+\gl^2e^{-2x})^{-1} \text{ for } i\ge 1, x,\gl\in\dR.
	\end{align}
	Expressing the $i^{\th}$ cumulant in terms of the empirical measures, we have that
	\begin{align*}
		\partial_x^i \log Z^{(j)}(x)=\sum_{k=1}^{N-j} f^i(x,\gl_{k,j}).
	\end{align*}
	It will be convenient to merge the sequence of zeroes corresponding to $Z^{(j)}$ and $Z^{(j+1)}$ into a single sequence. We use the convention that $\gc_{0,j}=-\infty$ and define the combined sequence $\gc_{k,j}$ such that
	\begin{align*}
		\gc_{k,j}
		:=\begin{cases}
			\gl_{(k+1)/2, j} & \text{ for } k \text{ odd} \\
			\gl_{k/2, j+1} & \text{ for } k \text{ even}.
		\end{cases}
	\end{align*}
	By the interlacing property~\eqref{def:interlace}, the combined sequence $\gc_{k,j}$ is non-decreasing in $k$, and symmetric about $0$. With these definitions we have
	\begin{align}\label{def:reindexinterlacesum}
		\begin{split}
			\abs{ \partial^i_x \log Z^{(j)}(x) -\partial_x^i\log Z^{(j+1)}(x) }
			& =\abs{ \sum_{l=0}^{N-j-1}\bigl( f^i\left(x,\gc_{2l+1,j} \right)-f^i\left(x,\gc_{2l,j} \right)\bigr) } \\
			&
			\le \sum_{k=0}^{2N-2j-1}\abs{ f^i\left(x,\gc_{k+1,j}\right)-f^i\left(x,\gc_{k,j}\right) }.
		\end{split}
	\end{align}
	Note that $\{-\infty,\gc_{0,j},\gc_{1,j},\gc_{2,j},\ldots,\gc_{2N-2j-1,j},\infty\}$ is a tagged partition of $(-\infty,\infty)$.	We aim to bound the variation of $f^i$ with respect to this partition. Fix $K<\infty$, it is easy to see that for $(x,\gl)\in [-K,K]\times \dR$, $|f^i|$ and $|\partial_{\gl}f^i|$ are both bounded. Furthermore for $1\le i\le 3$, $f^i$ decay to $0$ as $|\gl|\to {\infty}$, monotonically for sufficiently large $\gl$.
	Let $M_K$ be chosen such that for $|\gl|>M_K$, $f^i$ monotonically decays to $0$ as $|\gl|\to \infty$ for all $i$ and $x\in [-K,K]$. Note that $\gc_{N-j,j}=0$ by symmetry. We define
	\begin{align*}k^{*}:=\min\{k\mid \gc_{k,j}\ge M_K\}.
	\end{align*}
	By hypothesis, on $[M_K,\infty)$, $f^i$ is bounded and monotonically converges to $0$. Thus,
	\begin{align*}
		\sum_{k=k^{*}}^{2N-2j-1}\abs{ f^i\left(x,\gc_{k+1,j}\right)-f^i\left(x,\gc_{k,j}\right) }\le \norm{f^i}_{\infty}.
	\end{align*}
	Next, on $[0,M_K]$ we have that $\norm{\partial_{\gl}f}_{\infty}$ is finite and thus $f^i$ is of bounded variation. In particular, we have
	\begin{align*}
		\sum_{k=N-j}^{k^{*}}\abs{ f^i\left(x,\gc_{k+1,j}\right)-f^i\left(x,\gc_{k,j}\right) }\le M_K\norm{\partial_{\gl}f}_{\infty}.
	\end{align*}
	Combining the above bounds and using the symmetry about $0$, we finally get
	\begin{align*}
		\abs{ \partial^i_x \log Z^{(j)}(x) -\partial_x^i\log Z^{(j+1)}(x) }\le 2\norm{f}_{\infty} +2M_K\norm{\partial_{\gl}f}_{\infty}.
	\end{align*}
	We apply the triangle inequality to the following sum
	\begin{align*}
		\partial^i_x \log Z -\partial_x^i\log Z^{V'}
		=\sum_{j=0}^{|V'|} 
		\bigl(\partial^i_x \log Z^{(j)} -\partial_x^i\log Z^{(j+1)}\bigr),
	\end{align*}
	to complete the proof.
\end{proof}

When examining the variance of $\la U\ra$ as well as the influence of removal of edges on the joint cumulants of the number of unpaired vertices on disjoint blocks the following ratio bound will be useful.
\begin{lem}\label{lem:ratioderbound}We have $K_{i,V'}$ depending only on $i$, $|V'|$ and $x$ such that
	\begin{align*}\abs{ \partial_x^i(Z^{V'}/Z) }\le K_{i,V'}\abs{ (Z^{V'}/Z) }.\end{align*}

\end{lem}
\begin{proof}
	Note that, for all $i\ge 1$ we have
	\begin{align*}\partial^i_x(Z^{V'}/Z)
		=\partial_x^{i-1}\left( (Z^{V'}/Z)\partial_x\log (Z^{V'}/Z)\right)
		=\sum_{k=0}^{i-1}\begin{pmatrix} i-1 \\ k\end{pmatrix}\partial^k_x(Z^{V'}/Z)\cdot \partial^{i-k}_x\log (Z^{V'}/Z).
	\end{align*}
	For $i=1$, the Lemma follows immediately from Lemma~\ref{lem:vertexremovalbound} as
	\begin{align*}
		\abs{ \partial_x(Z^{V'}/Z) }=\abs{ (Z^{V'}/Z)\partial_x \log (Z^{V'}/Z) }\le C\abs{ (Z^{V'}/Z) }.
	\end{align*}
	We can now use induction along with Lemma~\ref{lem:vertexremovalbound} and Leibniz rule to complete the proof.
\end{proof}
The interlacing also enables us to examine the spatial statistics of the unpaired vertices. We exponentially tilt the measure so as to compute the cumulants of $U_{[1:k]}$ and $U_{[k+1:n]}$ for $1<k<n$. With $x_1$ and $x_{2}\in \dR$, we define the modified Hamiltonian:
\begin{align*}
	\cH_{x_1,x_{2}}(\fm):=\cH(\fm)+x_1U_{[1:k]}(\fm)+x_{2}U_{[k+1:n]}(\fm)
	,\end{align*}
the modified partition function
\begin{align*}
	Z_n(x_1,x_{2}):=\sum_{\fm\in \cM} \exp(\cH_{x_1,x_{2}}(\fm))
	,\end{align*}
and finally the tilted Gibbs measure
\begin{align*}
	\mu_{x_1,x_2}(\fm):=\frac{\exp(\cH_{x_1,x_{2}}(\fm))}{Z_n(x_1,x_{2})}
	.\end{align*}
As with the case of $Z_n(x)$ an alternative viewpoint is to regard the tilting as a re-weighting of the vertices, where the transformed weights are such that all vertices $u\in V_{\cG_{[1:k]}}$ carry the weight $x_1+\nu_u$ while all vertices $v \in V_{\cG_{[k+1:n]}}$ carry the weight $x_{2}+\nu_v$. The gauge transformation process is also identical, in order to pass the random vertex weights onto the edges. $Z_{[1:k]}$ and $Z_{[k+1:n]}$ will be denoted by $Z_1$ and $Z_{2}$ respectively for ease of notation. Thus,
\begin{align*}
	Z_1=Z_1(x_1) \text{ and } Z_{2}=Z_{2}(x_{2}).
\end{align*}

As introduced in~\eqref{def:Decomp}, $\cE_{k}$ denotes the collection of edges that bridge $\cG_{[1:k]}$ and $\cG_{[k+1:n]}$. We will use $A$ to denote an arbitrary subset of $\cE_{k}$. Let $Z_1^{A}$ denote the partition function on the principal subgraph of $\cG_{[1:k]}$ obtained on removal of all vertices incident to $A$. Recall the remainder term $R_{n,k}$ introduced in~\eqref{def:Decomp}, with notation $\cE_{k}$, $A$ and $Z^{A}$ as above:
\begin{align*}
	R_{n,k}=\log \left( \sum_{ A \subseteq \cE}\prod_{i : e_{k,i} \in A} \exp\left({\tilde{\go}_{k,i}-x_1-x_{2}}\right)\cdot \frac{Z^{A}_1}{Z_1}\cdot \frac{Z^{A}_{2}}{Z_{2}}\right).
\end{align*}
Showing that this error term is small at the level of the free energy is not difficult, the interlacing~\eqref{def:interlace} goes one step further in showing that the derivatives of the error up top order $3$ are small as well. The next lemma will make this precise
\begin{lem}\label{lem:cutbound}
	Let $Z_1$ and $Z_{2}$ be as above. Let $\mv{i}=(i_1,i_{2})$ denote a non-zero multi-index and $\partial_{\mv{x}}^{\mv{i}}:=\partial_{x_1}^{i_1}\partial_{x_{2}}^{i_{2}}$. There exists $C=C(x,\mv{i})<\infty$ such that
	\begin{align*}
	 |\partial_{\mv{x}}^{\mv{i}}R_{n,k}|\le C.\end{align*}
\end{lem}
As an immediate consequence of Lemma~\ref{lem:ratioderbound}, there exist constants $K_{i,A}$ such that for $l=1,2$, we have
\begin{align}
	\abs{ \partial^i_{x_{l}}(Z_{l}^{A}/Z_{l}) }\le K_{i,A}\cdot (Z_{l}^{A}/Z_{l}).
\end{align}
With the family of constants $K_{i,A}$ introduced in Lemma~\ref{lem:ratioderbound}, we define
\begin{align*}
	K_i:=\max_{A\subseteq \cE_{k}}\sum_{j=0}^iK_{i-j,A}.
\end{align*}
We now have all the ingredients required to prove Lemma~\ref{lem:cutbound}.

\begin{proof}[Proof of Lemma~\ref{lem:cutbound}]
	We aim to bound the derivatives of the error term
	\begin{align*}R_{n,k}=\log \left( 1+\sum_{A \neq \phi}\prod_{e \in A} \exp (\tilde{\go}_e-x_1-x_{2})\frac{Z^{A}_1Z^{A}_{2}}{Z_1Z_{2}}\right).\end{align*}
	We will be applying Lemma~\ref{lem:MomentCumulant} to
	\begin{align}
		r_{n,k}:=\exp(R_{n,k}).
	\end{align}
	Thus, it suffices to bound the ``moment'' terms
	\begin{align}\label{def:momentformula} m_{\mv{i}}
		:=r_{n,k}^{-1} \cdot \partial^{\mvi}_x r_{n,k}
		=r_{n,k}^{-1} \cdot \partial_x^{\mv{i}}\left(\sum_{A \neq \phi} \exp (\tilde{\go}_{A}-|A|x_1-|A|x_{2})\frac{Z^{A}_1Z^{A}_{2}}{Z_1Z_{2}}\right)
	\end{align}
	where we used $\tilde{\go}_{A}:=\sum_{e\in A}\tilde{\go}_e.$	We define
	\begin{align*}
		\Psi_{A,i,l}(x)
		:=\partial_x^i \left(e^{-|A|x}\cdot Z_{l}^{A}/Z_{l}\right)=\sum_{j=0}^i\begin{pmatrix} i \\ j \end{pmatrix}(-|A|)^{j}\cdot e^{-|A|x}\cdot \partial^{i-j}_x\frac{Z^{A}_{l}(x)}{Z_{l}(x)}.
	\end{align*}
	On applying Lemma~\ref{lem:ratioderbound}, we get
	\begin{align*}|\Psi_{A,i,l}|\le (2|A|)^i\frac{Z_{l}^{A}}{Z_{l}}\exp(-|A|x)\sum_{j=0}^iK_{i-j,A}\le (2h)^i\frac{Z_{l}^{A}}{Z_{l}}\exp(-|A|x)\sum_{k=0}^iK_{i-j,A}.\end{align*}
	We plug the bound in~\eqref{def:momentformula}, to get
	\begin{align*}
	 |m_{\mv{i}}| 
	 &\le (2h)^{i_1+i_{2}}K_{i_1}K_{i_{2}}\frac{\sum_{A\neq \phi} \exp(\tilde{\go}_{A}-|A|x_1-|A|x_{2})\frac{Z_1^{A}Z_{2}^{A}}{Z_1Z_{2}}}{1+\sum_{A \neq \phi} \exp (\tilde{\go}_{A}-|A|x_1-|A|x_{2})\frac{Z^{A}_1Z^{A}_{2}}{Z_1Z_{2}}}
 \le {(2h)}^{|i|}K_{i_1}K_{i_{2}}.
 \end{align*}
 This completes the proof.
\end{proof}
In Section~\ref{sec:Ulims} we will provide a sufficiency condition for the annealed variance of $\la U \ra$ to be bounded away from $0$. The localization of Lee-Yang zeroes to a compact interval when the random weights are bounded is key to this proof. It is convenient to prove a preliminary lemma for this now.

\begin{lem}\label{lem:derlowerbound}
	Let $\cG$ be a graph on $n$ vertices with bounded degree; the edges have bounded weights and the vertices have constant weight $x$. Let $u,v\in V(\cG)$, and $\cG^{u,v}$ be the graph obtained on removal of $u$ and $v$. Let $Z_1$ and $Z_{2}$ denote the corresponding partition functions and $M>0$ be such that the Lee-Yang zeroes are restricted to $[-M,M]$. We have
	\begin{align*}|\partial_x\log Z_1-\partial_x\log Z_{2}|\ge \left(\frac{e^{2x}-M^2}{e^{2x}+M^2}\right)_{+}.\end{align*}
\end{lem}
\begin{proof}
	Recall the definition of $f^1(x,\gl)$ from~\eqref{def:fidef}. Let $\{\gl_{k}\}$ and $\{\gl_{k}'\}$ be the Lee-Yang zeroes of $Z_1$ and $Z_{2}$. We have
	\begin{align}\label{eq:l451}
		|\partial_x\log Z_1-\partial_x \log Z_{2}|=\abs{ \sum_{k=1}^nf^1(x,\gl_{k})-\sum_{k=1}^{n-2}f^1(x,\gl_{k}') }.
	\end{align}
	It suffices to consider the case of $n$ being even, as if $n$ is odd, $Z_1$ and $Z_{2}$ have a common root located at $0$ which cancels on taking the difference. We may also sum exclusively over $\gl> 0$ since the roots are symmetric. We re-index the roots from $1$ to $l$ where $l=\frac{n}{2}$. By interlacing, we have that $\gl_{k}\le \gl_{k+2}'$ for all $k$. Rearranging, we get
	\begin{align*}
		\eqref{eq:l451} =
		\abs{ \frac{e^{2x}}{e^{2x}+(\gl_{l})^2} +\frac{e^{2x}}{e^{2x}+(\gl_{l-1})^2}-\frac{e^{2x}}{e^{2x}+(\gl_1')^2} + \sum_{k=1}^{l-2}\left( f^1(x,\gl_{k})-f^1(x,\gl_{k+2}')\right) }.
	\end{align*}
	Moreover, interlacing yields that
	\begin{align*}
	 \sum_{k=1}^{l-2}\left( f^1(x,\gl_{k})-f^1(x,\gl_{k+2}')\right)
		\ge 0.
	\end{align*}
	Combining the bounds
	\begin{align*}
	 \frac{e^{2x}}{e^{2x}+(\gl_{l})^2} +\frac{e^{2x}}{e^{2x}+(\gl_{l-1})^2}\ge \frac{2e^{2x}}{e^{2x}+M^2}
	 \text{ and }\frac{e^{2x}}{e^{2x}+(\gl_1')^2} \le 1
	\end{align*}
	 and applying triangle inequality completes the proof.
\end{proof}
%%%%%%%%%%%%%%%%%%%%%%%%%%%%%%%%%%%%%%%%%%%%%%%%
\section{Limit Theorems for the Number of Unpaired Vertices} \label{sec:Ulims}
%%%%%%%%%%%%%%%%%%%%%%%%%%%%%%%%%%%%%%%%%%%%%%%%
In Section~\ref{sec:LeeYang} we introduced exponential tilting in order to study the cumulants, as well as introduce the Lee-Yang zeroes. In this section, we will be relating the results of Section~\ref{sec:LeeYang} to the model at hand. As such, at the conclusion of calculations here, all parameters depending on $x$ will be evaluated at $x=0$. In the prior section we used the gauge invariance to make sure that the vertex weights were constant for the purpose of defining Lee-Yang zeroes. Here, we will freely move back and forth between the original partition function and the gauge transformed version. We may do this because the respective free energies differ by a constant, which vanishes on taking the derivative with respect to $x$.

%%%%%%%%%%%%%%%%%%%%%%%%%%%%%%%%%%%%%%%%%%%%%%%%%%%%%%%%%%
\subsection{Law of Large Numbers and Variance Convergence}
%%%%%%%%%%%%%%%%%%%%%%%%%%%%%%%%%%%%%%%%%%%%%%%%%%%%%%%%%%

The entirety of this section is dedicated to the proof of Theorem~\ref{thm:Mvar}, which is given at the end. We begin with proving the law of large numbers for the number of unpaired vertices for a typical matching $\fm$. The method of proof is entirely analogous to the convergence of the mean free energy. We will establish the convergence of the annealed average using subadditivity, a variance bound using Efron-Stein, and finally convergence of $\la U \ra_n$ using Chebyshev's inequality.

\begin{lem}\label{lem:Umean}
	For every $x\in\dR$, there exists $u(x)\in [0,1]$ such that
	\begin{align*}
		n^{-1}\E\la U \ra^x _n \to u(x) \end{align*}
\end{lem}

\begin{proof}
	We begin with the decomposition in~\eqref{def:Decomp} and take the derivative with respect to $x$
	\begin{align*}
		\partial_x \log Z_n =\partial_x \log Z_{[1:k]} +\partial_x \log Z_{[k+1:n]} + \partial_xR_{n,k}.
	\end{align*}
	Lemma~\ref{lem:cutbound} then yields
	\begin{align*}
		\E\la U\ra^x _n\le \E\la U\ra^x _{k} +\E\la U \ra^x _{n-k} + C(1).
	\end{align*}
	The subadditive lemma then yields the convergence, and we have
	\begin{align*}
		u(x):=\lim_{n \to \infty}n^{-1}\cdot \E \la U \ra^x_n
	\end{align*}
	exists for all $x$.
\end{proof}
Having established the correct centering, we move to the fluctuations. We establish the following upper bound on the variance.
\begin{lem}\label{lem:efron2} There exists a finite constant $C$ depending on $x$ such that
	\begin{align*}
		\var \la U\ra^x_n \le C n.
	\end{align*}
\end{lem}
\begin{proof}
	We use the same method as Lemma~\ref{lem:efron}, the only difference is that since we are taking a derivative with respect to $x$, the dependence of the $\ga$ and $\gb$ terms on $x$ needs to be taken into account. Let $e=(v,w)$ be an arbitrary edge and $u$ be an arbitrary vertex. With the same notation as Lemma~\ref{lem:efron},
	\begin{align*}
		\partial_x \log Z_n-\partial_x \log Z_n^{e}
		 & =\partial_x\int_{\go_e'}^{ \go_e}\frac{e^z}{e^z+{\ga_e}/{\gb_e}}\, dz
		 =-\partial_x\left(\frac{\ga_e}{\gb_e}\right)\cdot\int_{\go_e'}^{ \go_e} e^z\left(e^z+{\ga_e}/{\gb_e}\right)^{-2}\, dz.
	\end{align*}
	Analogously for the vertices, we have
	\begin{align*}
		\partial_x\log Z_n-\partial_x\log Z_n^{u}
		=-\partial_x\left( \frac{\hat{\alpha}_u}{\hat{\beta}_u}\right)\cdot \int_{\nu'_u}^{\nu_u} e^z\bigl(e^z+ {\hat{\alpha}_u}/{\hat{\beta}_u}\bigr)^{-2}\, dz.
	\end{align*}
	We illustrate the method for the variance that is contributed by the edge randomness. The vertex case is identical. By applying Lemma~\ref{lem:ratioderbound} and the triangle inequality we have
	\begin{align*}
		|\partial_x \log Z_n-\partial_x \log Z_n^{e}|\le \int_{\go_e\land \go_e'}^{\go_e\lor \go_e'}\frac{K_1e^z\cdot {\ga}/{\gb}}{\left(e^z+{\ga}/{\gb}\right)^2}\, dz.
	\end{align*}
	We now use the fact that when $a,b\ge 0$ then $2ab\le (a+b)^2$. Finally squaring, taking the expectation and summing over all edges complete the proof.
\end{proof}
The following is an easy corollary using Chebyshev's inequality.
\begin{cor}\label{cor:Uconv}
	For any $x\in \dR$, with $u(x)$ as defined in Lemma~\ref{lem:Umean}, we have
	\begin{align*}
		n^{-1}\la U\ra^x_n\toP u(x)\text{ as } n\to\infty.
	\end{align*}
\end{cor}

The convergence of $n^{-1}\la U\ra^x$ is enough to establish the convergence of all higher scaled cumulants, as it implies the weak convergence (in probability) of the sequence $\{\rho_n\}$.
\begin{definition}
	A probability measure $\mu$ on $\dR$ is said to be symmetric if $\mu_n[a,b]=\mu_n [-b,-a]$ for all $a\le b$ in $\dR$.
\end{definition}
Clearly, the family of empirical measures $\{\rho_n\}$ is symmetric. We recall that $n^{-1}\la U\ra_n$ is a linear statistic in the the zeroes. If we let $z$ denote $e^{2x}$ to have
\begin{align*}
	n^{-1}\la U\ra_n^x=\int\frac{z}{z+\gl^{2}}\, d\rho_n(\gl).
\end{align*}
The transform with respect to the function $z/(z+\gl^{2})$ determines the measure. We prove this below.
\begin{lem}\label{lem:doublelaplace}
	Let $\{\mu_n\}_{n\ge 1}$ be a tight sequence of symmetric probability measures on the real line. Consider the following transform of $\mu_n$:
	\begin{align*}
		F_n(z)=\int_{\dR}\frac{z}{z+\gl^{2}}\, d\mu_n(\gl), \text{ } z\in (0,\infty).
	\end{align*}
	If $F_n(z)$ converges pointwise to $F(z)$ as $n\to\infty$, then $\mu_n$ is weakly convergent to a probability measure $\mu$.
\end{lem}
\begin{proof}
	The symmetry is of crucial importance here. Let $L_n$ denote the random variable with distribution $\mu_n$. Symmetry of $\mu_n$ implies that the distribution of $L_n$ can be recovered from the distribution of $L_n^{2}$. The transform $F_n(z)$ is given by
	\begin{align*}
		F_n(z)= \E\frac{z}{z+L_n^{2}}=z\E\int_{0}^{\infty}\exp(-t(z+L_n^{2}))dt, { }z>0.
	\end{align*}
	Applying Fubini's theorem, we obtain that
	\begin{align*}
		F_n(z)=z\int_{0}^{\infty} \exp(-zt)\Gamma_n(t)dt,\text{ }z\ge 0
	\end{align*}
	where
	\begin{align*}
		\Gamma_n(t):=\E\exp(-tL_n^{2})=\int_{\dR}e^{-\gl^{2}t} d\mu_n(\gl),\text{ }t\ge0
	\end{align*}
	is the Laplace transform of the measure $\mu_n$. Since the sequence $\{\mu_n\}$ is tight, there exists a subsequence $\{\mu_{n_{k}}\}$ weakly converging to a probability measure $\mu$. If $\mu_n$ is not weakly convergent, there must be another subsequence $\{\mu_{m_{k}}\}$ with a distinct weak limit $\mu^{*}$. Since $\exp({-zt})$ is bounded and continuous, we have that $\Gamma_{n_{k}}(t)$ and $\Gamma_{m_{k}}(t)$ both converge pointwise as $k \to \infty$ to limits, say $\Gamma(t)$ and $\Gamma^{*}(t)$, respectively. By our original hypothesis, $F_{n_{k}}(z)$ and $F_{m_{k}}(z)$ both converge to $F(z)$ as $k \to \infty$. Applying dominated convergence (we may since $|\Gamma_n|\le 1$), we obtain that
	\begin{align*}
		\int_{0}^{\infty}\exp(-zt) (\Gamma(t)-\Gamma^{*}(t))dt=0,\text{ }\forall z>0.
	\end{align*}
	Uniqueness of the Laplace transform now implies that $\Gamma=\Gamma^{*}$, which in turn implies that $\mu=\mu^{*}$. Thus, $\mu_n$ is weakly convergent with limit $\mu$.
\end{proof}
\begin{cor}
	The sequence of empirical distributions $\rho_n$ weakly converges to a probability measure $\rho$ as $n\to\infty$, in probability.
\end{cor}
We thus have a representation for the limiting density of unpaired vertices. With $f^1$ as defined in~\eqref{def:fidef} we have
\begin{align}\label{def:vertexdensity}
	u(x)=\int_{\dR} f^1(x,\gl) d\rho(\gl).
\end{align}
We have $u:=u(0)$. With the mean behavior characterized, we move to the asymptotic behavior of the variance.

\begin{lem}\label{lem:Uvar}
	We have $\sigma_{Q}(x)$ and $\sigma_{A}(x) \in [0,\infty)$ such that pointwise in $x$
	\begin{align*}
		n^{-1}\la\widehat{U}^{2} \ra^x_n\toP \sigma_{Q}^{2}(x)
		\text{ and }
		n^{-1}\var(\la U\ra^x_n) \to \gs^{2}_{A}(x).
	\end{align*}
\end{lem}

\begin{proof}
	We establish the convergence of the quenched variance first. We have
	\begin{align*}
		n^{-1}\la {\widehat{U}}^2\ra^x _n
		= \int_{\dR} f^1(x,\gl)\left(1 -f^1(x,\gl)\right)\, d\rho_n(\gl).
	\end{align*}
	The integrand is bounded and continuous. Convergence of the variance is an immediate corollary of the weak convergence in probability of $\rho_n$, \ie\
	\begin{align*}
		n^{-1}\la {\widehat{U}}^2\ra^x _n
		\toP
		\gs_Q^2(x)
		:=\int_{\dR} f^1(x,\gl)\left(1 -f^1(x,\gl)\right)\, d\rho(\gl).
	\end{align*}
	The argument for the convergence of the annealed variance is similar to the proof of the convergence of variance of the free energy in Lemma~\ref{lem:zvarconv}. We write the decomposition
	\begin{align*}\partial_x \log Z_n=\partial_x\log Z_{[1:k]}+\partial_xZ_{[k+1:n]} + R_{n,k}^1\end{align*}
	and denote
	\begin{align*}
		\gs_{A,n}^2(x)=\var(\partial_x\log Z_n).
	\end{align*}
	Computing the variance of both sides and applying Cauchy-Schwarz inequality, we have
	\begin{align*}\gs_{A,n,}^2(x)\le \gs_{A,k,}^2(x)+\gs_{A,n-k}^2(x)+K_1(\gs_{A,k}(x)+\gs_{A,n-k}(x)).\end{align*}
	By Lemma~\ref{lem:efron2}, $\gs^{2}_{A,n}(x)\le C n$.
	As an immediate corollary of the subadditivity lemma, we get that $n^{-1}{\gs^{2}_{A,n}(x)}$ converges as $n\to\infty$, we denote the limit as $\gs^{2}_{A}$.
\end{proof}
To go back to our original model, we have $\sigma_{Q}:=\sigma_{Q}(0)$ and $\sigma_{A}:=\sigma_{A}(0)$. As with the free energy, we now establish lower bounds for $\gs_{Q}$ and $\gs_{A}$ to verify non degeneracy of the limiting law.

\begin{cor}\label{cor:QUvarlowerbound}
	There exists a constant $C>0$ such that with probability approaching $1$,
	\begin{align*}
		\bigl \la (U-\la U\ra_n)^{2}\bigr\ra_n \ge Cn. 
	\end{align*}
\end{cor}
\begin{proof}
Note that, we have
\[
 n^{-1}\bigl \la (U-\la U\ra_n)^{2}\bigr\ra_n 
 = \int 2\gl^2(1+\gl^2)^{-2}\, d\rho_n(\gl).
\]
	From Lemma~\ref{lem:zeroanticonc}, we may find $\eps,\gd>0$ such that with probability approaching 1, $\rho_n[0,\gd)\le 1-\eps$. The tightness proved in Lemma~\ref{lem:tightness} says that given $\eps>0$, we may find an $K$ such that with probability approaching 1, $\rho_n{[K,\infty)}\le \eps/2$. We now use the fact that $\rho_{n}([\gd,K])\ge \eps/2$ to complete the proof. 
\end{proof}

The method to prove the lower bound for the annealed variance $\gs_{A}$ is analogous to Lemma~\ref{lem:zvarlowerbound}.
\begin{lem}\label{lem:uannealedvarlb}
	Let both the vertex and edge weights be compactly supported in $[-K,K]$ for fixed $0<K<\infty$, such that the Lee-Yang zeroes of $Z_n$ are localized to $[-M,M]$ for some fixed $0<M<\infty$ depending on $K$. For $e^x\ge M+\eps$ for some fixed $\eps>0$ there exists $C=C(\eps)> 0$ such that
	\begin{align*}
		\var \la U\ra_n^x \ge C n.
	\end{align*}
\end{lem}
\begin{proof}
	The proof, as with the case of $\gs_{F}$ is about extracting the bounds on the influence of a single edge. The martingale argument then yields the lower bound. Recall the filtration $\cF_j$ defined in Lemma~\ref{lem:zvarlowerbound}. With the same notation, we have
	\begin{align*}
		\partial_x \log W_n-\partial_x \log W_n^{j}=\partial_x \int_{\go_j'}^{\go_j }\frac{e^z}{e^z+{\ga_{e_j}}/{\gb_{e_j}}}\, dz.
	\end{align*}
	We now apply the triangle inequality and Lemma~\ref{lem:derlowerbound} to obtain
	\begin{align*}
		\abs{ \partial_x\log W_n-\partial_x\log W^{(j)}_n }
		 & = \abs{ \partial_x \left(\frac{\ga_{e_j}}{\gb_{e_j}}\right) }\int_{\go_j^{'}\land \go_j}^{\go_j'\lor \go_j} \frac{e^z}{\left(e^z + {\ga_{e_j}}/{\gb_{e_j}} \right)^2}\, dz \\
		 & \ge \left(\frac{e^{2x}-M^2}{e^{2x}+M^2}\right)_{+} \int_{\go_j^{'}\land \go_j}^{\go_j' \lor \go_j} \frac{e^z}{\left( e^z+{\ga_{e_j}}/{\gb_{e_j}}\right)^2}\, dz \\
		 & \ge \left(\frac{e^{2x}-M^2}{e^{2x}+M^2}\right)_{+}\int_{\go_j^{'}\land \go_j}^{\go_j'\lor \go_j}\frac{e^z}{(e^z+\exp(2x)+(1+M)^{d_{max}})^{2}}\, dz.
	\end{align*}
	Note that the integrand may be rewritten as $y^{-1}\sech^{2}\left(\frac{z+y}{2}\right)$ where \[y=\exp(2x)+(1+M)^{d_{max}}.\]
	For the fixed compact interval $[-K,K]$ which our weights are restricted to, the integrand is uniformly bounded below in $z$. Applying this lower bound, squaring and then taking the expectation completes the result. The non transitive case follows in a procedure identical to that followed in Lemma~\ref{lem:zvarlowerbound}.
\end{proof}
\begin{cor}\label{cor:UAvar}
	Let the support of $\go$ be contained in $[a,b]$ and the support of $\nu$ be contained in $[c,d]$ such that $b-2c<-\log d_{max}$. Then we have a constant $C>0$ such that
	\begin{align*}
		\var \la U\ra_n\ge Cn.
	\end{align*}
\end{cor}
\begin{proof}
	Let $e=(u,v)$ be an edge. The gauge transformed weight of $e$ is $\tilde{\go}_e=\go_e-\nu_u-\nu_v$. By hypothesis, this implies that with probability $1$,
	\begin{align*}
		\tilde{\go}_e<-\log d_{max}.
	\end{align*}
	Thus for any vertex $v$,
	\begin{align*}
		\Delta_v=\sum_{e\sim v}\exp(\tilde{\go}_e)<1
	\end{align*}
	This implies that there is an $M<1$ such that all Lee-Yang zeroes associated to $Z_n(x)$ are located in $[-M,M]$. Lemma~\ref{lem:uannealedvarlb} now applies with $x=0.$
\end{proof}

\subsubsection{Proof of Theorem~\ref{thm:Mvar}}

Here we combine the previous results proved in this section, to complete the proof of Theorem~\ref{thm:Mvar}. Convergence of $n^{-1}\la U\ra_{n}$ follows from Corollary~\ref{cor:Uconv}, by taking $x=0$. Variance bounds follow from Lemma~\ref{lem:Uvar}, Corollaries~\ref{cor:QUvarlowerbound} and~\ref{cor:UAvar}. \qed

%%%%%%%%%%%%%%%%%%%%%%%%%%%%%%%%%%%%%%%%%%%%%%%%%%%%%%%%%%
\subsection{Central Limit Theorems} 
%%%%%%%%%%%%%%%%%%%%%%%%%%%%%%%%%%%%%%%%%%%%%%%%%%%%%%%%%%

In the previous subsection we established the asymptotic mean behavior of $U$, proved convergence of the respective variances and provided sufficiency conditions for their non degeneracy. We are now ready to prove both the quenched and annealed central limit theorems.

\subsubsection{Proof of Theorem~\ref{thm:uclt}}
Consider the quenched moment generating function $\Gamma^{\widehat{U}}_n(\xi,x)$ of $\widehat{U}$ given by
\begin{align*}
	\Gamma^{\widehat{U}}_n(\xi,x):=\bigl\la \exp(\xi\cdot n^{-\half} \widehat{U})\bigr\ra^x _n
	\text{ for }
	\xi\in \dR \text{ fixed}.
\end{align*}
We have,
\begin{align*}
	\Gamma^{\widehat{U}}_n(\xi,x)
	 & =\frac{1}{Z(x)}\sum_{\fm \in \cM}\exp\left( \xi\cdot n^{-\half} \widehat U(\fm)+x U(\fm)\right)\prod_{e \in \fm}\exp(\tilde{\go}_e)\\
	&=\exp \bigl( n^{-\half}\xi\cdot \la U\ra^x _n \bigr)\cdot {Z\bigl(x+n^{-\half}\xi\bigr)}/{Z(x)}.
\end{align*}
Clearly,

\begin{align*}\log \Gamma^{\widehat{U}}_n(\xi,x)=\log Z\bigl(x + n^{-\half}\xi\bigr)-\log Z(x)-n^{-\half}\xi \cdot \partial_x \log Z(x).\end{align*}
Using Taylor's theorem upto second order, we get
\begin{align*}\log \Gamma^{\widehat{U}}_n(\xi,x)= n^{-1} \xi^{2} \cdot\partial_x^{2}\log Z(x )+\cR(\xi,x,n),\end{align*}
where the error term $\cR$ can be written in terms of the third derivative with respect to $x$
\begin{align*}
	|\cR|\le n^{-\half} \cdot |\xi|^{3}\int \abs{ \partial_x^2 (1+\gl^2e^{-2x})^{-1}}\, d\rho_n(\gl).
\end{align*}
It suffices to consider $\xi$ from some fixed compact interval. The integrand is a bounded function in both $x$ and $\gl$, therefore the error decays to $0$ in the limit. We have that
\[
 \log \Gamma^{\widehat{U}}_n(\xi)\toP \frac{1}{2}\gs_{Q}^2(x)\xi^{2} \text{ for all }\xi.
\]
This is exactly the cumulant generating function of the Gaussian with variance $\sigma_{Q}^{2}(x)$. Weak convergence is guaranteed by the fact that we have $\bigl\la \cosh {n^{-\half}\widehat{U}}\bigr\ra^x _n$ bounded in probability, which implies that the laws of $\widehat{U}$ are tight in probability. Uniqueness follows from the uniqueness of the Laplace transform. To conclude, we take $x=0$.\qed\\

We move towards the joint Central Limit Theorem. For this purpose we need to use the exponential tilting introduced for Lemma~\ref{lem:cutbound}. Thus, $Z=Z(x_1,x_{2})$, $Z_{[1:k]}=Z_{[1:k]}(x_1)$ and $Z_{[k+1:n]}=Z_{[k+1:n]}(x_{2})$. Let $\mv{x}$ denote $(x_1,x_{2})^{T}.$

\subsubsection{Proof of Theorem~\ref{thm:jointclt}}

In complete analogy to proof of the Central Limit Theorem, we now examine the behavior of the MGF of the vector $n^{-\half}\widehat{\mv{U}}$, that is
\begin{align*}
	\Gamma^{\widehat{\mv{U}}}_n(\mv{\xi},\mv{x})=\left\la \exp(\mv{\xi}\cdot n^{-\half}\mv{\widehat{U}}) \right\ra^{\mv{x}} _n
\end{align*}
where
$
	\mv{\xi}:=(\xi_1,\xi_{2})^{T}\text{ and }\mv{U}:=(U_{[1:k]},U_{[k+1:n]})^{T}.
$
Taking the logarithm of $G_n$ and Taylor expanding, we get
\begin{align*}
	 & \Gamma^{\widehat{\mv{U}}}_n(\mv{\xi},\mv{x}) \\
	 & \quad =\log Z \bigl(x_1+n^{-\half}\xi_1,x_{2}+n^{-\half}\xi_{2}\bigr)-\log Z(x_1,x_{2})-n^{-\half}\xi_1\la U_{[1:k]}\ra^{\mv{x}}_n - n^{-\half}\xi_{2}\la U_{[k+1:n]}\ra^{\mv{x}}_n \\
	 & \quad
	=n^{-1}\cdot \mv{\xi}^{T} \mv{\gS}(n) \mv{\xi} +n^{-3/2}\cR(\xi_1,\xi_{2},\nu).
\end{align*}
Here, $\mv{\gS}(n)$ is the matrix of second derivatives, and has the form \begin{align*}\mv{\gS}_n=\diag(\partial_x^{2}\log Z_{[1:k]},\partial_x^{2} \log Z_{[k+1,n]})+D^{2}R_{n,k}\end{align*}
where $D^{2}R_{n,k}$ denotes the Hessian matrix of the remainder term $R_{n,k}$ with respect to $x_1$ and $x_{2}$. We may apply Lemma~\ref{lem:cutbound} to conclude that \begin{align*}n^{-1}\norm{D^{2}R_{n,k}}\toP 0\end{align*} for any choice of matrix norm $\norm{\cdot}$. Theorem~\ref{thm:Mvar} can then be applied to yield
\begin{align*}
	n^{-1}\mv{\gS}_n\toP \sigma_{Q}^{2}\cdot \diag(t,1-t)
	.\end{align*}
As for the remainder term, the pure third derivatives can be bounded as per the same argument in theorem~\ref{thm:uclt}. As for the mixed derivatives, Lemma~\ref{lem:cutbound} yields that $|\partial_x^{\mv{i}}\log Z|\le C(i_1,i_{2})$ since the mixed derivative of the separated free energies is zero.
In the limit, we have that pointwise for all $\mv{x}$
\begin{align*}
	\log \Gamma^{\widehat{\mv{U}}}_n \toP \mv{\xi}^{T} \mv{\gS} \mv{\xi}.
\end{align*}
This establishes the joint Central Limit Theorem, as the tightness can be established again via an argument analogous to Theorem~\ref{thm:uclt}. We conclude with setting $\mv{x}=0$.\qed\\

Moving to the annealed Central Limit Theorem for $\la U\ra_n$, we must proceed via the same dyadic subdivison route as in Theorem~\ref{thm:Zclt}.

\subsubsection{Proof of Theorem~\ref{thm:uhatclt}}
	The procedure is identical to the proof of Theorem~\ref{thm:Zclt}, all that is required is to systematically replace each $Z$, $R$ and $T$ with the respective $x$ derivatives and apply Lemma~\ref{lem:cutbound}. We therefore omit a complete proof.\qed

%%%%%%%%%%%%%%%%%%%%%%%%%%%%%%%%%%%%%%%%%%%
\section{Limiting Height Function and Brownian Motion}\label{sec:BM}
%%%%%%%%%%%%%%%%%%%%%%%%%%%%%%%%%%%%%%%%%%%
In this section, we characterize the limiting behavior of the height function $\theta_n(t)=U_{[1:\lfloor nt\rfloor]}$. This is towards understanding the structure of a typical matching $\fm$ chosen with respect to $\mu$. We first establish the mean behavior, and then characterize the fluctuations. The law of large numbers results (both quenched and annealed) combine to yield that
\begin{align*}
 \pr\left(\abs{ \theta_n(t)-\E\theta_n(t) }\ge \eps n\right)\to 0.
\end{align*}
The limiting height function $\theta(t)$ is given by \begin{align*}
 \theta(t)=\lim_{n\to \infty} {\theta_n(t)}/{n},\text{ }t\in [0,1].
\end{align*} 
We begin this section by characterizing $\theta$.

\begin{lem}\label{lem:linearGrowth}
	Let $u$ be as given in~\eqref{def:vertexdensity}. There exists a finite constant $C$ such that for all $m\in \dN$
	\begin{align*}|\E\la U\ra _{m}- m\cdot u|\le C\end{align*}

\end{lem}
\begin{proof}

	Recall the dyadic subdivision as introduced in Section~\ref{section:MFE} along with the accompanying notation, and let $n=\sum_{i=0}^{l}a_i2^i$. Recall from~\eqref{def:dyadic}
	\begin{align*}
		\log Z_n=\sum_{v\in V_{k}} \log Z^{v}_{\pi_{k}(\mva)}+\sum_{j=1}^k\bigl(\sum_{v\in V_{j-1}}(R_{j-1}^{v}+a_jT_{j-1}^{v})\bigr).
	\end{align*}
	Taking the derivative, the expectation, and then absolute value of both sides,
	\begin{align*}
		\abs{ \E\partial_x\log Z_n -\E \sum_{v\in V_{k}} \partial_x\log Z^{v}_{\pi_{k}(\mva)} }
		\le \sum_{j=1}^k\sum_{v\in V_j}\E\abs{ \partial_x R_{j-1}^{v} +a_j\partial_xT_{j-1}^{v} }.
	\end{align*}
	Applying Lemma~\ref{lem:cutbound}, we get
	\begin{align*}
		\abs{ \E\partial_x\log Z_n -\E \sum_{v\in V_{k}} \partial_x\log Z^{v}_{\pi_{k}(\mva)}}
		\le C\sum_{i=1}^k2^i=C'2^k.
	\end{align*}
	Dividing both sides by $n$ and simplifying, we have
	\begin{align*}
		\abs{ \frac{1}{n}\E\partial_x\log Z_n- \frac{2^k}{n}\E \partial_x\log Z_{m} }\le C'\frac{2^k}{n}
	\end{align*}
	where $m=\pi_{k}(\mva)$. Let $p=l-k$ and
	\begin{align*}
		\gc_{n,p}:=1-\frac{2^{p}m}{n}.
	\end{align*}
	It is easy to show that
	\begin{align*}
		|\gc_{n,p}|\le \frac{2}{m}\le 1.
	\end{align*}
	Multiplying the entire expression by $m$ and applying the triangle inequality, we get

	\begin{align*}
		\abs{ m\frac{1}{n}\E\partial_x\log Z_n- \E \partial_x\log Z_{m} }-|\gc_{n,p}|\E\partial_x\log Z_{m}\le C'.
	\end{align*}
	Let
	\begin{align*}
		|\gc_n|:=\sup_{p}|\gc_{n,p}|\le \frac{2}{m}.
	\end{align*}
	Rearranging the terms and using the fact that $0\le \partial_x \log Z_{m}\le m$ (there are at most $m$ unpaired vertices), we obtain

	\begin{align*}
		\abs{ m\cdot \frac{\E \partial_x\log Z_n}{n} - \E \partial_x \log Z_{m} }\le C' +2.
	\end{align*}
	Note that letting $p\to \infty$ implies that $n \to \infty$ while keeping $m$ fixed. Thus, we have
	$|m\cdot u-\E\la U\ra _{m}|\le C'+2.$
	This completes the proof.
\end{proof}
\begin{cor}
	The limiting height function $\theta(t)$ is given by $\theta(t)=u\cdot t$
\end{cor}
With the centering term calculated, we are now ready to characterize the limiting fluctuations about ${\theta}(t)$.

\subsection{Proof of Theorem~\ref{thm:BM}}
	Let $0\le t_1<t_{2}< \ldots<t_{k}\le 1$. To verify the distributional convergence, we need to establish three properties
	\begin{enumerate}
		\item\textbf{Starting at Zero}: $\widehat{\theta}_n(0)=0$
		\item \textbf{Joint Normality}: $\widehat{\theta}_n(t_{i+1})-\widehat{\theta}_n(t_i) \tod \N(0,(t_{i+1}-t_i)\gs^{2})$ for all $1\le i<k$
		\item \textbf{Independent Increments}: $\{\widehat{\theta}_n(t_{i+1})-\widehat{\theta}_n(t_i)\}_{i=1}^{k-1}$ are asymptotically jointly independent.
	\end{enumerate}
	For simplicity,we consider the case $k=3$ with two disjoint time intervals. The method of proof can easily be adapted to establish the general case. The first property is trivial. We will prove the latter two together. Consider the random vector
	\begin{align*}
		\mv{\Theta}_n=n^{-\half}\left({\widehat{\theta}}_n(t_{3})-{\widehat{\theta}}_n(t_{2}),{\widehat{\theta}}_n(t_{2})-{\widehat{\theta}}_n(t_1) \right)^T.
	\end{align*}
	The next step is applied to both components and is illustrated with the first component. Observe that
	\begin{align*}
		{\widehat{\theta}}(t_{3})-{\widehat{\theta}}(t_{2})=
		n^{-\half}\left(U_{[1:\lfloor t_{2}n\rfloor]}-U_{[1:\lfloor t_{3}n\rfloor] }-n (t_{2}-t_{3})u\right).
	\end{align*}
	Lemma~\ref{lem:linearGrowth} tells us that $n(t_{2}-t_{3})u$ can be replaced by $\E \la U \ra _{\lfloor t_{3}n\rfloor-\lfloor t_{2}n \rfloor}$ by incurring at most a constant error, which we denote by $R_{t_{2},t_{3},n}$. Further, $U_{[1:\lfloor t_{3}n\rfloor]}-U_{[1:\lfloor t_{2}n\rfloor]}=U_{[\lfloor t_{2}n \rfloor :\lfloor t_{3}n\rfloor]}$ So, we are left with
	\begin{align*}
		n^{-\half}\left(U_{[\lfloor t_{2}n \rfloor:\lfloor t_{3}n \rfloor]}-\E \la U \ra _{\lfloor t_{3}n\rfloor-\lfloor t_{2}n \rfloor}\right)+n^{-\half}{R_{t_{2},t_{3},n}}.
	\end{align*}
	We now separate into quenched and annealed components as
	\begin{align*}
		n^{-\half}\left(U_{[\lfloor t_{2}n \rfloor:\lfloor t_{3}n\rfloor]}-\la U_{[\lfloor t_{2}n \rfloor:\lfloor t_{3}n\rfloor]}\ra _n\right) + n^{-\half}\left(\la U_{[\lfloor t_{2}n \rfloor:\lfloor t_{3}n\rfloor]}\ra _n-\E \la U \ra _{\lfloor t_{3}n\rfloor-\lfloor t_{2}n \rfloor}\right)+n^{-\half}R_{t_{2},t_{3},n}.
	\end{align*}
	Applying Lemma~\ref{lem:cutbound} we can replace $\la U_{[\lfloor t_{2}n \rfloor:\lfloor t_{3}n\rfloor]}\ra _n$ with $\la U\ra _{\lfloor t_{3}n\rfloor -\lfloor t_{2} n \rfloor}$ by incurring a constant bounded error, which we absorb into $R_{t_{2},t_{3},n}$. Correspondingly, we may also replace $\la U_{[\lfloor t_1n \rfloor:\lfloor t_{2}n\rfloor]}\ra _n$ with $\la U\ra _{\lfloor t_{2}n\rfloor -\lfloor t_1 n \rfloor}$. Since the weights on their respective sections are independent, it follows that $\la U\ra _{\lfloor t_{3}n\rfloor -\lfloor t_{2} n \rfloor}$ and $\la U\ra _{\lfloor t_{2}n\rfloor -\lfloor t_1 n \rfloor}$ are independent. We now separate $\mv{\Theta}$ into quenched and annealed components. More precisely, we define
	\begin{align*}
		\mv{\Theta}^Q_n
		 & :=n^{-\half}\left(U_{[\lfloor t_{2}n \rfloor:\lfloor t_{3}n\rfloor]}-\la U_{[\lfloor t_{2}n \rfloor:\lfloor t_{3}n\rfloor]}\ra _n,U_{[\lfloor t_1n \rfloor:\lfloor t_{2}n \rfloor]}-\la U_{[\lfloor t_1n \rfloor:\lfloor t_{2}n\rfloor]}\ra _n\right)^{T} \\
		\mv{\Theta}^{A}_n
		 & :=n^{-\half} \left(\la U\ra _{\lfloor t_{3}n\rfloor -\lfloor t_{2} n \rfloor}-\E\la U\ra _{\lfloor t_{3}n\rfloor -\lfloor t_{2} n \rfloor} ,\la U\ra _{\lfloor t_{2}n\rfloor -\lfloor t_1 n \rfloor}-\E\la U\ra _{\lfloor t_{2}n\rfloor -\lfloor t_1 n \rfloor}\right)^{T} \\
		\text{and }
		\mv{R}_n
		 & :=n^{-\half}(R_{t_{2},t_{3},n},R_{t_1,t_{2},n}).
	\end{align*}
	Clearly,
	\begin{align*}
		\mv{\Theta}_n=\mv{\Theta}^Q_n +\mv{\Theta}^{A}_n + \mv{R}_n.
	\end{align*}
	Let
	\begin{align*}
		\mv{\xi}= (\xi_1,\xi_{2})^{T},\text{ } \mv{\zeta}=(\zeta_1,\zeta_{2})^{T} \text{ and } \mv{\gc}=(\gc_1,\gc_{2})^{T}\in \dR^{2}.
	\end{align*}
	We work with the characteristic function
	\begin{align*}
		\Phi_n\left( \mv{\xi},\mv{\zeta},\mv{\gc} \right):=\E \exp \left( \sqrt{-1}\mv{\xi}\cdot \mv{\Theta}^Q_n +\sqrt{-1}\mv{\zeta}\cdot \mv{\Theta}^{A}_n +\sqrt{-1}\mv{\gc}\cdot \mv{R}_n\right).
	\end{align*}
	Observe that
	\begin{align*}
		 & |\E\exp(\sqrt{-1}\mv{\xi}\cdot \mv{\Theta}_n^{Q} +\sqrt{-1}\mv{\zeta}\cdot \mv{\Theta}_n^{A}+\sqrt{-1}\mv{\gc}\cdot \mv{R}_n)-\E\exp(\sqrt{-1}\mv{\xi}\cdot \mv{\Theta}_n^{Q} +\sqrt{-1}\mv{\zeta}\cdot \mv{\Theta}_n^{A})| \\
		 & \qquad\le \E|1-\exp\sqrt{-1}\mv{\gc}\cdot \mv{R}_n|
		\le \norm{\mv{\gc}\cdot \mv{R}_n}_{\infty} \le n^{-\half}\cdot C \cdot |\mv{\gc}|.
	\end{align*}
	Thus, for all values of $\mv{\gc}$,
	\begin{align}\label{def:BMerror}
		|\Phi_n(\mv{\xi},\mv{\zeta},\mv{\gc})-\Phi_n(\mv{\xi},\mv{\zeta},0)|\to 0 \text{ as } n\to \infty.
	\end{align}
	We may effectively ignore the remainder term and drop the argument $\mv{\gc}$ from $\Phi^{}$. We tackle the quenched term $\mv{\Theta}_n^{Q}$ next. It is clear that
	\begin{align*}
		\Phi_n\left( \mv{\xi},\mv{\zeta} \right)
		=
		\E\bigl( \left\la \exp(\sqrt{-1} \mv{\xi}\cdot \mv{\Theta}^{Q}_n) \right\ra \exp(\sqrt{-1} \mv{\zeta}\cdot\mv{\Theta}^{A}_n)\bigr).
	\end{align*}
	Let
	\begin{align*}
		\mv{\gS}^{Q}
		:=\gs^{2}_{Q}\cdot\diag(t_3-t_2, t_2-t_1).
	\end{align*}
	By Theorem~\ref{thm:jointclt} the following is uniformly bounded and
	\begin{align*}
		\abs{ \left\la \exp(\sqrt{-1} \mv{\xi}\cdot \mv{\Theta}^Q_n) \right\ra - \exp (-\mv{\xi}^{T} \mv{\gS}^{Q} \mv{\xi}) }\toP 0 \text{ as }n\to \infty.\end{align*}
	Thus the dominated convergence theorem tells us that
	\begin{align}\label{def:BMquenched}
		|\Phi_n(\mv{\xi},\mv{\zeta})-\exp(-\mv{\xi}^{T} \mv{\gS}^{Q} \mv{\xi})\E\exp (\sqrt{-1}\mv{\zeta}\cdot \mv{\Theta}_n^{A})|\to 0.
	\end{align}
	To find the limit of $\Phi_n$, the final step is to evaluate the limit of
	$
		\E \exp( \sqrt{-1}\mv{\zeta}\cdot \mv{\Theta}^{A}_n).
	$
	Let
	\begin{align*}
		\mv{\gS}^{A}
		=\gs^{2}_{A}\cdot\diag(t_3-t_2, t_2-t_1).
	\end{align*}
	By construction, the components of the annealed vector are independent of each other, and by Theorem~\ref{thm:uhatclt} we have
	\begin{align} \label{def:BMannealed}
		\E \exp(\sqrt{-1} \mv{\zeta} \cdot \mv{\Theta}^{A}_n) \to \exp(-\mv{\zeta}^{T}\mv{\gS}^{A} \mv{\zeta}).
	\end{align}
	Combining equations~\eqref{def:BMerror},~\eqref{def:BMquenched} and~\eqref{def:BMannealed}, we conclude that
	\begin{align*}
		\Phi_n(\mv{\xi},\mv{\zeta},\mv{\gc})
		\to
		\exp(-\mv{\xi}^{T}\mv{\gS}^{Q} \mv{\xi} -\mv{\zeta}^{T} \mv{\gS}^{A}\mv{\zeta}).
	\end{align*}
	What we have proved is that in distribution in probability $\mv{\Theta}^{Q}_n$ converges to a jointly Gaussian vector $\mv{\Theta}^{Q}$ with covariance matrix $\mv{\gS}^{Q}$, $\mv{\Theta}^{A}_n$ converges to jointly Gaussian vector $\mv{\Theta}^{A}$ with covariance matrix $\mv{\gS}^{A}$, moreover $\mv{\Theta}^{Q}$ and $\mv{\Theta}^{A}$ are independent of each other.
	Thus, $\mv{\Theta}^{Q}+\mv{\Theta}^{A}$ is a jointly Gaussian vector with covariance matrix
	$
		(\gs_{Q}^{2}+\gs_{A}^{2})\cdot \diag(t_{2}-t_{3},t_1-t_{2})
	$
	which verifies that the increments of $\widehat{\theta}$ are Gaussian with the correct covariance structure.\qed 

%%%%%%%%%%%%%%%%%%%%%%%%%%%%%%%%%%%%%%%%%%%%%%%%%%%

%%%%%%%%%%%%%%%%%%%%%%%%%%%%%%%%%%%%%%%%%%%
\section{Connection to Jacobi Operators}\label{sec:Jacobi}
%%%%%%%%%%%%%%%%%%%%%%%%%%%%%%%%%%%%%%%%%%%
The case where $H$ is a singleton is interesting as the monomer-dimer model becomes exactly solvable. To see this, note that the recurrence relation for the partition function may be written as
\begin{align*}
	Z_{n+1}=\exp({\nu_{n+1}})Z_n+\exp({\go}_{n,n+1})Z_{n-1}
\end{align*}
This is a three step recurrence, which means we may write it in determinant form. Let $\go_{n,n+1}$ be denoted as $\go_n$. Consider the following Jacobi matrix,
\begin{align}\label{def:Jacobi}
	\mv{A}_n:=
	\begin{pmatrix}
		\sqrt{-1}e^{\nu_1} & e^{ \go_1/2} & \hdots & 0 \\ e^{\go_1/2} & \ddots & \ddots & \vdots \\ \vdots & \ddots & \sqrt{-1}e^{\nu_{n-1}} & e^{\go_{n-1}/2} \\ 0 & \hdots & e^{\go_{n-1}/2} & \sqrt{-1}e^{\nu_n}
	\end{pmatrix}=\mv{\gO}_n +\mv{V}_n,
\end{align}
where $\mv{\gO}_n$ is the weighted adjacency matrix of the graph, and $\mv{V}_n$ is the diagonal part of $\mv{A}_n$ which may be interpreted as an on sit potential. Just by expanding the determinant along the last row we have that
\begin{align*}
	\det \mv{A}_{n+1}:=\sqrt{-1}e^{\nu_n}\det\mv{A}_{n-1}-e^{\go_n}\det\mv{A}_{n-1}
\end{align*}
It is easy to show that the phase of $\det \mv{A}_n$ is periodic; in fact we have that
\begin{align*}
	\arg \det\mv{A_n}=\exp\left(\sqrt{-1}\cdot {n\pi}/{2}\right) \text{ for all }n.
\end{align*}
One can easily check that $Z_n$ and $|\det\mv{A}_n|$ satisfy the same recurrence, with the same initial conditions. Thus, $Z=|\det \mv{A}_n|$. The determinant structure is very useful for explicit computation. For instance, the probability that vertex $k$ is unpaired is given by $e^{\nu_{k}}[\mv{A}_n^{-1}]_{k,k}$.
The gauge transformation as seen in Lemma~\ref{lem:gauge} can also be cast into matrix form, by means of rescaling the rows and coulumns. The transformation $\nu_i \to 0$, $\go_i\to \go_i-\nu_i$ and $\go_{i-1}\to \go_{i-1}-\nu_i$ maybe realized as
\begin{align}\label{def:matrixgauge}
	\mv{A}_n\to \mv{D}_n(i)\mv{A}_n\mv{D}_n(i),
\end{align}
where $\mv{D}_n(i)$ is a diagonal matrix with $[D_n(i)]_{jj}=1$ for all $j\neq i$ and $[D_n(i)]_{ii}=\exp(-\nu_i/2)$. Clearly the order in which the transformations are applied is irrelevant as the diagonal matrices commute with each other. The exponential tilting of the model can be implemented by multiplying only the diagonal entries of $\mv{A}_n$ by $e^x$. If we are tilting the model so as to shift the wight of vertex $i$ by a factor of $x$, this may be achieved via the linear transformation
\begin{align*}
	\mv{A}_n\to \mv{A}_n+\sqrt{-1}(e^x-1)\mv{\Pi}_i\mv{A}_n\mv{\Pi}_i
\end{align*}
where $\mv{\Pi}_i$ denotes the diagonal matrix with $[\mv{\Pi}_i]_{ii}=1$ and all other entries $0$. It is clear to see that the tilting operations and the gauge transformations commute with each other and thus the order in which they are applied is irrelevant. It is interesting to note that the gauge operations employed here may be regarded as a change of inner product so as to transform $\mv{A}_n$ to a normal matrix.
Now consider the tilted, gauge transformed matrix, which we denote as $\tilde{\mv{A}}_n(x)$. Explicitly, this is given by
\begin{align*}
	\tilde{\mv{A}}_n(x)
	=
	\begin{pmatrix}
		\sqrt{-1}e^x & e^{ \tilde{\go}_1/2} & \hdots & 0 \\
		e^{\tilde{\go}_1/2} & \ddots & \ddots & \vdots \\
		\vdots & \ddots & \sqrt{-1}e^x & e^{\tilde{\go}_{n-1}/2} \\
		0 & \hdots & e^{\tilde{\go}_{n-1}/2} & \sqrt{-1}e^x
	\end{pmatrix}
	=\tilde{\mv{\gO}}_n+\sqrt{-1}e^x\mv{I}_n.
\end{align*}
It is now immediately clear what the Lee-Yang zeroes of the partition function are, the $\gl$ are exactly the eigenvalues of $\tilde{\mv{\gO}}$. Our analysis can be rephrased in the language of spectral theory. In particular, the interlacing and boundedness of the Lee-Yang zeroes all have analogous forms in the study of eigenvalues of Jacobi matrices. The Gibbs average of the number of unpaired vertices may be given by an expression related to the resolvent of $\tilde{\mv\gO}_n$
\begin{align*}\la U\ra_n=e^{2x}\tr\left[(\tilde{\mv\gO}^{2}+e^{2x}\mv{I}_n)^{-1}\right].
\end{align*}
The limiting free energy of the monomer-dimer model can be related to the Lyapunov exponent corresponding to $\tilde{\mv{\gO}}_n$, denoted as $\gc(x)$, in a form analogous to the Thouless formula. We have that
\begin{align*}
	\lim_{n\to \infty} \frac{1}{n}\log Z_n=\gc(0)-\E\nu.
\end{align*}
Our central limit theorem for the free energy may also be interpreted as a central limit theorem for the Lyapunov exponent, and the central limit theorem for the number of unpaired vertices as that for the resolvent. For more details about the spectral theory, we refer to~\cite{GerT1}.

%%%%%%%%%%%%%%%%%%%%%%%%%%%%%%%%%%%%%%%%%%%
\section{Further Questions}\label{sec:disc}
%%%%%%%%%%%%%%%%%%%%%%%%%%%%%%%%%%%%%%%%%%%

\subsection{CLT for Determinants of Random Band Matrices}

It is not hard to adapt the methods for cylinder graphs for the so-called $h$-band graphs, \ie\ the line graph where vertices $i$ and $j$ are adjacent iff $|i-j|\le h$. As per Section~\ref{sec:Jacobi}, a natural question is if one can apply these methods to the \textit{band matrices}, weighted adjacency matrices of band graphs.
\begin{definition}
	An $N\times N$ matrix $\mv{A}$ is said to be a band matrix with band $h$ if $[\mv{A}]_{ij}=0$ for all $|i-j|>h$. In particular, Jacobi matrices are band matrices with band $1$.
\end{definition}

We do not have a similar combinatorial interpretation of the $h$--band determinant. However, it is interesting to see if one can identify the correct scaling and centering and then prove a Gaussian central limit theorem for the $\log$ determinant of $h$--band matrices, using the methods of Theorem~\ref{thm:Zclt}. The issue that arises is that the error control required will be far more subtle as the signs of the permutations have to be taken into account, a problem entirely bypassed in the tridiagonal case.

\subsection{Correlation Structure}

An important question yet to be answered is that of the correlation structure and the precise decay. We avoided analyzing the correlation by comparing our partition function to the product of smaller partition functions and showed that the error is small. From the study of random tridiagonal operators, especially in the context of Anderson Localization, the correlation structure of the matching at the local level can be carried out in the case of $|H|=1$. However, the matrix structure is absent in all other cases. Characterizing the correlation between unpaired vertices explicitly and analyzing the decay is an essential next step. Analysis of the $k$-point correlation structure would also help strengthen our result for Brownian motion convergence. We have proved convergence in the sense of finite-dimensional distributions; however, tightness at the process level remains open.

\subsection{Higher Dimensional Lattices}

Several of our results, most notably the variance bounds and the tightness of the Lee-Yang zeroes in the disordered setting, extend to more general classes of graphs. The bounded degree assumption is the only restriction. However, the subadditivity that we used to establish convergence of mean free energy and the unpaired vertex density is absent in any situation other than the 1-dimensional case described here. Any attempt to extend these results to higher-dimensional lattices such as $\dZ^2$ will require an alternative method to prove the convergence of the mean free energy. One method is to try and work with the $d$--dimensional box of side length $2^n$, which we denote as $B_n$, and show that $\log Z_{B_n}$ satisfies a subadditivity condition. However, the error control is highly non-trivial in this case, as using a bound like ours results in errors of the same order as that of the partition function.

%\subsection{A Disordered Perfect Matching Problem}
%
%Consider the two dimensional box $B^{2}_n:=[0,n]\times [0,n]$ and the following scheme of weights; the edges connecting $(i,j)$ to $(i,j+1)$ have weight $\go_j$ and the edges connecting $(i,j)$ to $(i+1,j)$ have weights given by independent copies $\go_i$. The weighted adjacency matrix with Kasteleyn orientation is given by
%\begin{align*}
%	\mv{K}_n=\mv{A}_n\otimes \mv{I}_n +\sqrt{-1}\mv{I}_n\otimes \mv{A}_n.
%\end{align*}
%The partition function $Z_n$ of the pure dimer model on $B^{2}_n$ is given by $\sqrt{|\det \mv{K}|}$. It is easy to show the convergence of the mean free energy, especially for bounded weights, as an extension of our methods. An Efron-Stein argument can be used to show that the variance of $Z_n$ is bounded above by $n^{3}$. We expect a central limit theorem to hold as well. It will also be interesting to analyze the structure of the height function and characterize the limiting behavior. Due to the long-range correlations of the weights, we expect long-range correlations in the height function.

%%%%%%%%%%%%%%%%%%%%%%%%%%%%%%%%%%%%%%%%%%%
\vskip.1in
\noindent{\bf Acknowledgments.}
We would like to thank Felix Christian Clemen, Gayana Jayasinghe, Grigory Terlov and Qiang Wu for enlightening discussions.

\bibliography{1dm.bib}

\begin{thebibliography}{10}

\bibitem{AC2}
Diego Alberici and Pierluigi Contucci.
\newblock Solution of the monomer-dimer model on locally tree-like graphs.
  {R}igorous results.
\newblock {\em Comm. Math. Phys.}, 331(3):975--1003, 2014.

\bibitem{AC1}
Diego Alberici, Pierluigi Contucci, Micaela Fedele, and Emanuele Mingione.
\newblock Limit theorems for monomer-dimer mean-field models with attractive
  potential.
\newblock {\em Comm. Math. Phys.}, 346(3):781--799, 2016.

\bibitem{AC3}
Diego Alberici, Pierluigi Contucci, and Emanuele Mingione.
\newblock A mean-field monomer-dimer model with randomness: exact solution and
  rigorous results.
\newblock {\em J. Stat. Phys.}, 160(6):1721--1732, 2015.

\bibitem{BIKS}
M.~Biskup, C.~Borgs, J.~T. Chayes, L.~J. Kleinwaks, and R.~Koteck\'{y}.
\newblock Partition function zeros at first-order phase transitions: a general
  analysis.
\newblock {\em Comm. Math. Phys.}, 251(1):79--131, 2004.

\bibitem{CHAT}
Sourav Chatterjee.
\newblock A general method for lower bounds on fluctuations of random
  variables.
\newblock {\em Ann. Probab.}, 47(4):2140--2171, 2019.

\bibitem{SPIN}
David Cimasoni and Nicolai Reshetikhin.
\newblock Dimers on surface graphs and spin structures. {I}.
\newblock {\em Comm. Math. Phys.}, 275(1):187--208, 2007.

\bibitem{CDG}
C.~D. Godsil.
\newblock Matching behaviour is asymptotically normal.
\newblock {\em Combinatorica}, 1(4):369--376, 1981.

\bibitem{HAM}
J.~M. Hammersley.
\newblock Generalization of the fundamental theorem on sub-additive functions.
\newblock {\em Proc. Cambridge Philos. Soc.}, 58:235--238, 1962.

\bibitem{HL}
Ole~J. Heilmann and Elliott~H. Lieb.
\newblock Theory of monomer-dimer systems.
\newblock {\em Comm. Math. Phys.}, 25:190--232, 1972.

\bibitem{CLT}
D.~Iagolnitzer and B.~Souillard.
\newblock Lee-{Y}ang theory and normal fluctuations.
\newblock {\em Phys. Rev. B (3)}, 19(3):1515--1518, 1979.

\bibitem{MKJ}
Mark Jerrum.
\newblock Erratum: ``{T}wo-dimensional monomer-dimer systems are
  computationally intractable'' [{J}. {S}tatist. {P}hys. {\bf 48} (1987), no.
  1-2, 121--134; {MR}0914432 (89d:82008)].
\newblock {\em J. Statist. Phys.}, 59(3-4):1087--1088, 1990.

\bibitem{Kast2}
P.~W. Kasteleyn.
\newblock Dimer statistics and phase transitions.
\newblock {\em J. Mathematical Phys.}, 4:287--293, 1963.

\bibitem{KAST1}
P.W. Kasteleyn.
\newblock The statistics of dimers on a lattice: I. the number of dimer
  arrangements on a quadratic lattice.
\newblock {\em Physica}, 27(12):1209--1225, 1961.

\bibitem{KRS}
Claire Kenyon, Dana Randall, and Alistair Sinclair.
\newblock Approximating the number of monomer-dimer coverings of a lattice.
\newblock {\em J. Statist. Phys.}, 83(3-4):637--659, 1996.

\bibitem{Rken1}
Richard Kenyon.
\newblock Dominos and the {G}aussian free field.
\newblock {\em Ann. Probab.}, 29(3):1128--1137, 2001.

\bibitem{LEB}
J.~L. Lebowitz, B.~Pittel, D.~Ruelle, and E.~R. Speer.
\newblock Central limit theorems, {L}ee-{Y}ang zeros, and graph-counting
  polynomials.
\newblock {\em J. Combin. Theory Ser. A}, 141:147--183, 2016.

\bibitem{ROB}
John~Keith Roberts.
\newblock The adsorption of hydrogen on tungsten.
\newblock {\em Proceedings of the Royal Society A: Mathematical, Physical and
  Engineering Sciences}, 152(876):445--463, 1935.

\bibitem{ROS}
Haskell~P. Rosenthal.
\newblock On the subspaces of {$L^{p}$} {$(p>2)$} spanned by sequences of
  independent random variables.
\newblock {\em Israel J. Math.}, 8:273--303, 1970.

\bibitem{TF}
H.~N.~V. Temperley and Michael~E. Fisher.
\newblock Dimer problem in statistical mechanics---an exact result.
\newblock {\em Philos. Mag. (8)}, 6:1061--1063, 1961.

\bibitem{GerT1}
Gerald Teschl.
\newblock {\em Jacobi operators and completely integrable nonlinear lattices},
  volume~72 of {\em Mathematical Surveys and Monographs}.
\newblock American Mathematical Society, Providence, RI, 2000.

\bibitem{PERM}
L.~G. Valiant.
\newblock The complexity of computing the permanent.
\newblock {\em Theoret. Comput. Sci.}, 8(2):189--201, 1979.

\bibitem{LEE1}
C.~N. Yang and T.~D. Lee.
\newblock Statistical theory of equations of state and phase transitions. {I}.
  {T}heory of condensation.
\newblock {\em Phys. Rev. (2)}, 87:404--409, 1952.

\end{thebibliography}
\bibliographystyle{plain}
\end{document}